\pdfoutput=1\relax
\documentclass[oneside]{amsart}
\usepackage{geometry}

\usepackage{amsthm,thmtools} 
\usepackage{environ}
\usepackage{todonotes}
\usepackage{stix}

\makeatletter

\NewEnviron{restatethis}[2]{%
  \begin{#1}\label{#2}\BODY
  \protected@write\@auxout{}{%
    \string\@restatetheorem{#1}{#2}{\csname the#1\endcsname}{\detokenize\expandafter{\BODY}}%
  }%
  \end{#1}%
}
\newcommand{\@restatetheorem}[4]{%
  \expandafter\gdef\csname beginrestatethis@#2\endcsname{%
    \begingroup
    \expandafter\def\csname the#1\endcsname{#3}%
    \begin{#1}%
  }
  \expandafter\gdef\csname bodyrestatethis@#2\endcsname{#4}%
  \expandafter\gdef\csname endrestatethis@#2\endcsname{%
    \end{#1}%
    \endgroup
  }
}

\NewEnviron{reformulatetheorem}[1]{
    \csname beginrestatethis@#1\endcsname%
    \BODY
    \csname endrestatethis@#1\endcsname%
    \addtocounter{thm}{-1}
}

\newcommand{\restate}[1]{%
  \begingroup
  \@skiphyperreftrue 
  \let\label\@gobble 
  \csname beginrestatethis@#1\endcsname%
  \csname bodyrestatethis@#1\endcsname%
  \csname endrestatethis@#1\endcsname%
  \endgroup
  \addtocounter{thm}{-1}
}
\makeatother

\usepackage{afterpage}
\usepackage{float}
\usepackage{spectralsequences}

\usepackage{varioref} 
\usepackage{hyperref}
\usepackage[noabbrev,nameinlink,capitalize]{cleveref}

\usepackage{stmaryrd} 

\usepackage{scalerel} 

\hypersetup{
    colorlinks,
    linkcolor={red!30!black},
    citecolor={blue!50!black},
    urlcolor={blue!80!black}
}

\declaretheorem[name=Theorem, numberwithin=section]{thm}
\declaretheorem[name=Lemma, sibling=thm]{lem}
\crefname{lem}{lemma}{lemmas}
\declaretheorem[name=Proposition, sibling=thm]{prop}
\declaretheorem[name=Corollary, sibling=thm]{cor}
\declaretheorem[name=Not A Corollary, sibling=thm]{notcor}

\theoremstyle{definition}
\declaretheorem[name=Definition, sibling=thm]{defn}

\makeatletter
\renewcommand\subsection{\@startsection{subsection}{2}%
  \z@{-.5\linespacing\@plus-.7\linespacing}{.5\linespacing}%
  {\normalfont\scshape}}
\makeatother

\input{macros}

\tikzcdset{arrow style=math font}

\begin{document}
\title[An Orientation Map for Height $p-1$ Real $E$ Theory]{An Orientation Map for Height $\boldsymbol{p-1}$ Real $\boldsymbol{E}$ Theory}
\author{Hood Chatham}
\date{\today}
\email{hood@mit.edu}

\begin{abstract}
\let\mylangle\mysmalllangle
\let\myrangle\mysmallrangle

Let $p$ be an odd prime and let $\EO = E_{p-1}^{hC_p}$ be the $C_p$ fixed points of height $p-1$ Morava $E$ theory.
We say that a spectrum $X$ has algebraic $\EO$ theory if the splitting of $K_*(X)$ as an $K_*[C_p]$-module lifts to a topological
splitting of $\EO \sm X$.   We develop criteria to show that a spectrum has algebraic $\EO$ theory, in particular showing that any
connective spectrum with mod $p$ homology concentrated in degrees $2k(p - 1)$ has algebraic $\EO$ theory.  As an application, we answer a question posed by Hovey and Ravenel \cite{hovey-ravenel} by producing a unital orientation $\MY_{4p-4}\to \EO$ analogous to the $\MSU$ orientation of $\KO$ at $p=2$.
\end{abstract}
\maketitle

\section{Introduction}
Let $E$ be a spectrum equipped with a unit map $S^0\to E$. A sphere bundle $V\colon Z\to \BGLS$ has a Thom spectrum $\Th(V)$ which comes with a unit map $S^0\to \Th(V)$. An $E$-orientation of the bundle $V$ is a choice of unital map $\Th(V)\to E$. If $V$ can be written as a pullback of a sphere bundle $W\colon Y\to \BGLS$, then there is a natural unital map $\Th(V)\to \Th(W)$ so an $E$-orientation of $W$ restricts to an $E$-orientation of $V$.

One strategy to understand $E$-orientations of bundles is to find an $E$-orientable bundle that is as universal as possible. We can then show that some other bundle is $E$-orientable by expressing it as the pullback of this ``universal'' orientable bundle. For instance, the map $\BSU\to \BGLS$ is $\KO$-orientable, so any bundle $V\colon Z\to \BGLS$ that factors through the map $\BSU\to \BGLS$ is orientable. This means that any sphere bundle that comes from a complex vector bundle with vanishing first Chern class is $\KO$-orientable.  Similarly, the map $\BU[6]\to \BGLS$ is $\TMF$-orientable so any sphere bundle that comes from a complex vector bundle with vanishing first two Chern classes is $\TMF$-orientable. The localizations $L_{K(1)}\KO$ and $L_{K(2)}\TMF$ are the $p=2$ and $p=3$ cases of a family of cohomology theories called higher real $E$-theories $\EO_{p-1}$. Since $\BSU = \BU[4]$ is the $4$-connective cover of $\BU$ and $\BU[6]$ is the $6$-connective cover of $\BU$, it is natural to guess that there might be an $\EO$-orientation of $\BU[2p]$. However, the standard map $\BU[2p]\to \BGLS$ is not $\EO$-orientable when $p>3$ according to an observation of Hovey \cite[Proposition 2.3.2]{hovey}.

We prove that the canonical bundle over the Wilson space $Y_{4p-4}$ is $\EO$-orientable. The Wilson space $Y_{2k}$ is obtained by starting with a $p$-local even dimensional sphere and attaching even cells to kill odd homotopy classes \cite{wilson}. The resulting spaces have even homotopy groups and torsion free even integral homology groups. Each Wilson space is an infinite loop space of $\BPn<n>$ for some appropriate $n$, for instance $Y_{4p-4}=\OS{\BPtw}{4p-4}$ is the $(4p-4)$th loop space of $\BPtw$ \cite{wilson}. The space $\BU[2p]$ has an Adams splitting
\begin{align*}
\BU[2p]
&\simeq \OS{\BPo}{2p}\times\cdots \times \OS{\BPo}{4p-4}\\
&= Y_{2p}\times Y_{2p+2}\times \OS{\BPo}{2p+4} \times \cdots \times \OS{\BPo}{4p-4}
\end{align*}
$\BU[2p]$ does not have even cohomology because $\OS{\BPo}{2k}$ doesn't have even cohomology when $k > p + 1$. We think of the Wilson space $Y_{4p-4}=\OS{\BPtw}{4p-4}$ as an even replacement for $\OS{\BPo}{4p-4}$. Hovey and Ravenel \cite{hovey-ravenel} computed the Adams Novikov spectral sequence for the Thom spectrum $MY_{4p-4}$ of the standard map $Y_{4p-4}\to BU$ through a range and observed that it looked like several copies of the homotopy fixed point spectral sequence for $\EO$. Because of this, they asked whether there could be a unital orientation map $MY_{4p-4}\to \EO$. We answer their question by showing that such a map exists:
\begin{restatethis}{thm}{4p-4_orientation}
Let $f\colon Y_{4p-4}\to \BGLS$ be any map. There is an equivalence $\EO\sm \Mf\simeq \EO\sm Y_{(4p-4)+}$ of $\EO$-modules, so there is a map of spectra $\Mf\to \EO$ which factors the unit map $S^0\to \EO$.
\end{restatethis}

As a replacement for an orientation map $\MU[2p]\to \EO$ we obtain an orientation map $M\OS{\MU}{2p}\to \EO$ (\Cref{MU2p-orientation}).

Our goal is to prove that certain bundles are $\EO$-orientable. Characteristic classes determine an easily computed obstruction to orientability. Given a cohomology theory $E$ and a space $Z$ we say that $E$-orientability of complex bundles over $Z$ is \emph{Chern determined} if the condition that $V$ is an $E$-orientable bundle over $Z$ is equivalent to some algebraic congruences on the Chern classes $c_i(V)\in H^{2i}(Z)$. If $E$-orientability of bundles over $Z$ is Chern determined we can easily determine which bundles over $Z$ are $E$-orientable.

Consider the case $E=\KO$. The mod $2$ reduction of the first Chern class $c_1(V)\in H^2(Z)$ determines the $\eta$ attaching map into the zero cell in $\Th(V)$. Since the zero cell is split in $\susp^{\infty}_+Z$ and $\eta$ is detected in $\KO_*$, a necessary condition for a bundle $V$ to be $\KO$-orientable is that $c_1(V)= 0\pmod{2}$. This is the only obstruction to $\KO$ orientability visible to Chern classes so a space $Z$ has Chern-determined $\KO$-orientability if every bundle $V$ over $Z$ such that $c_1(V)= 0\pmod{2}$ is $\KO$-orientable. An application of a theorem of Bousfield (\Cref{even-KO-module-splitting}) implies that any even space has Chern-determined $\KO$-orientability. The space $\BSU$ is even and $4$-connected, so this implies that every complex vector bundle over $\BSU$ is $\KO$-orientable. This proves \cref{4p-4_orientation} in the case that $p=2$ and $f$ factors through $\BU$.

In the odd prime case we have analogously that $\alpha_1\in \pi_{2p-3}(\EO)$ is nonzero. The $\alpha_1$ attaching maps in a space $Z$ are detected by the $P^1$ action on the mod $p$ cohomology. This implies that if a bundle $V$ over $Z$ is $\EO$-orientable, we must have $P^1(u)=0$ where $u$ is the Thom class of $V$ in $\HFp_*(Z)$. In the case of the universal bundle over $\BU$, $P^1(u)=\overline{\psi}_{p-1}u$ where $\overline{\psi}_{p-1}$ is the $(p-1)$st power sum characteristic class reduced mod $p$. Therefore, if $V$ is orientable then $\psi_{p-1}(V)\in H^{2p-2}(\BU)$ must be divisible by $p$. Analogously to the case when $p=2$, this is the only obstruction to orientability visibile to Chern classes so a space $Z$ has Chern-determined $\EO$-orientability if every bundle $V$ over $Z$ with $\psi_{p-1}(V)=0\pmod{p}$ is $\EO$-orientable. We show that every space with cohomology concentrated in degrees divisible by $2p-2$ has Chern-determined $\EO$-orientability. In particular, $Y_{4p-4}$ satisfies this sparsity condition and is sufficiently connective that $\psi_{p-1}$ lives in a zero group. This implies the odd prime case of \cref{4p-4_orientation} when $f$ factors through $\BU$. The case when $f$ is a general sphere bundle requires a bit of extra care with terminology but is fundamentally the same.

\subsection*{Background}
Fix an odd prime $p$. All spectra are implicitly $p$-completed. Let $E=E_{p-1}$ be the Morava $E$-theory corresponding to the Honda formal group law of height $p-1$ over $\F_{p^{p-1}}$. Let $\m$ be the maximal ideal of $E_*$ and let $K_* = E_*/\m = \F_{p^{p-1}}[u^{\pm}]$. The Morava stabilizer group at height $p-1$ contains elements of order $p$. Let $G$ be a maximal finite subgroup of $\G$ containing some element of order $p$. Such a subgroup is unique up to conjugacy. Let $\EO = E^{hG}$. For an $\EO$-module $M$ write $\EEO_*(M) = \pi_*(E\sm_{\EO} M)$. A more detailed review of the facts that we need about the Morava stabilizer group appears at the beginning of \Cref{sec-E_*X_l}. Bujard \cite{finite-subgroups} has completely classified finite subgroups of the Morava stabilizer group.

\def\nonhurewiczcolor{nofill}
\def\hurewiczcolor{yesfill}
\sseqset{nofill/.style={fill=none},yesfill/.style={fill=black}}

\sseqset{Zclass/.style={rectangle}, pZclass/.style={rectangle,fill=none}}
\DeclareSseqCommand\betaclass {d()} {
    \IfNoValueF{#1}{\pushstack(#1)}
    \class[\hurewiczcolor](\lastx+10,\lasty+2)
    \structline[\hurewiczcolor]
}

\begin{sseqdata}[
    name=EO_3,
    y range={0}{8},x range={0}{72},
    x tick step=5, y tick step=2,
    classes={fill,inner sep=1pt},
    x grid step=5,
    y grid step=2,
    class labels={black,pin},
    classes=\nonhurewiczcolor,struct lines=\nonhurewiczcolor,
    width=\textwidth-15pt,
    height=\textheight/3-25pt,
    title={$p=3$}
    ]

\begin{scope}[classes=\hurewiczcolor,struct lines=\hurewiczcolor]
\class[Zclass](0,0)
\class["\alpha_1" above](3,1)
\structline
\class["\beta_1" below](10,2)
\structline(\lastclass2)
\class(13,3)
\structline
\structline(\lastclass2)
\class(20,4)
\structline(\lastclass2)
\class(30,6)
\structline
\class(40,8)
\structline
\class(37,3)
\structline
\end{scope}

\class(27,1)
\structline
\structline(30,6)

\class[pZclass](24,0)
\class[pZclass](48,0)
\end{sseqdata}

\DeclareSseqCommand\betaclass {d()} {
    \IfNoValueF{#1}{\pushstack(#1)}
    \class[\hurewiczcolor](\lastx+38,\lasty+2)
    \structline[\hurewiczcolor]
}

\let\alphastructstyle\empty
\NewSseqCommand\betaclasspair {ss} {
    \IfBooleanTF{#1}{
        \let\bottomcolor\hurewiczcolor
    }{
        \let\bottomcolor\empty
    }
    \IfBooleanTF{#2}{
        \let\bottomcolortwo\hurewiczcolor
    }{
        \let\bottomcolortwo\empty
    }
    \class[\hurewiczcolor](\lastx1+38,\lasty1+2)
    \structline[\hurewiczcolor](\lastclass2)
    \class[\bottomcolor](\lastx1+38,\lasty1+2)
    \structline[\bottomcolor,\alphastructstyle]
    \structline[\bottomcolortwo](\lastclass2)
}

\NewSseqCommand\triplebetapair {} {
    \betaclasspair
    \betaclasspair
    \betaclasspair
}

\begin{sseqdata}[
    name=EO_5,
    y range={0}{32},x range={0}{650},
    xscale=0.045,
    yscale=0.4, x tick step=100, y tick step=5,
    classes={fill,inner sep=1pt},
    x grid step=50,
    y grid step=2,
    classes=\nonhurewiczcolor,struct lines=\nonhurewiczcolor,
    class labels={black,pin},
    width=\textwidth-15pt,
    height=2*\textheight/3-60pt,
    title={$p=5$}
    ]

\begin{scope}[classes=\hurewiczcolor,struct lines=\hurewiczcolor]
\class[Zclass](0,0)
\class(7,1)
\structline
\betaclasspair
\betaclasspair
\betaclasspair
\betaclass(\lastclass1)

\betaclass
\end{scope}
\def\alphastructstyle{densely dashed}
\class(167,1)
\structline[densely dashed]
\betaclasspair*
\betaclasspair**
\betaclasspair**

\betaclass(\lastclass1)
\class(327,1)
\structline[densely dashed]
\betaclasspair
\betaclasspair*
\betaclasspair**
\let\alphastructstyle\empty

\betaclass(\lastclass1)
\class(487,1)
\structline
\betaclasspair
\betaclasspair
\betaclasspair*

\foreach \v in {1,...,4}{
    \class[pZclass](160*\v,0)
}

\classoptions["\alpha_1" above](7,1)
\classoptions["\beta_1" below](38,2)
%
\end{sseqdata} 
\afterpage{
    \clearpage
    \raggedbottom
    \begin{figure}[H]
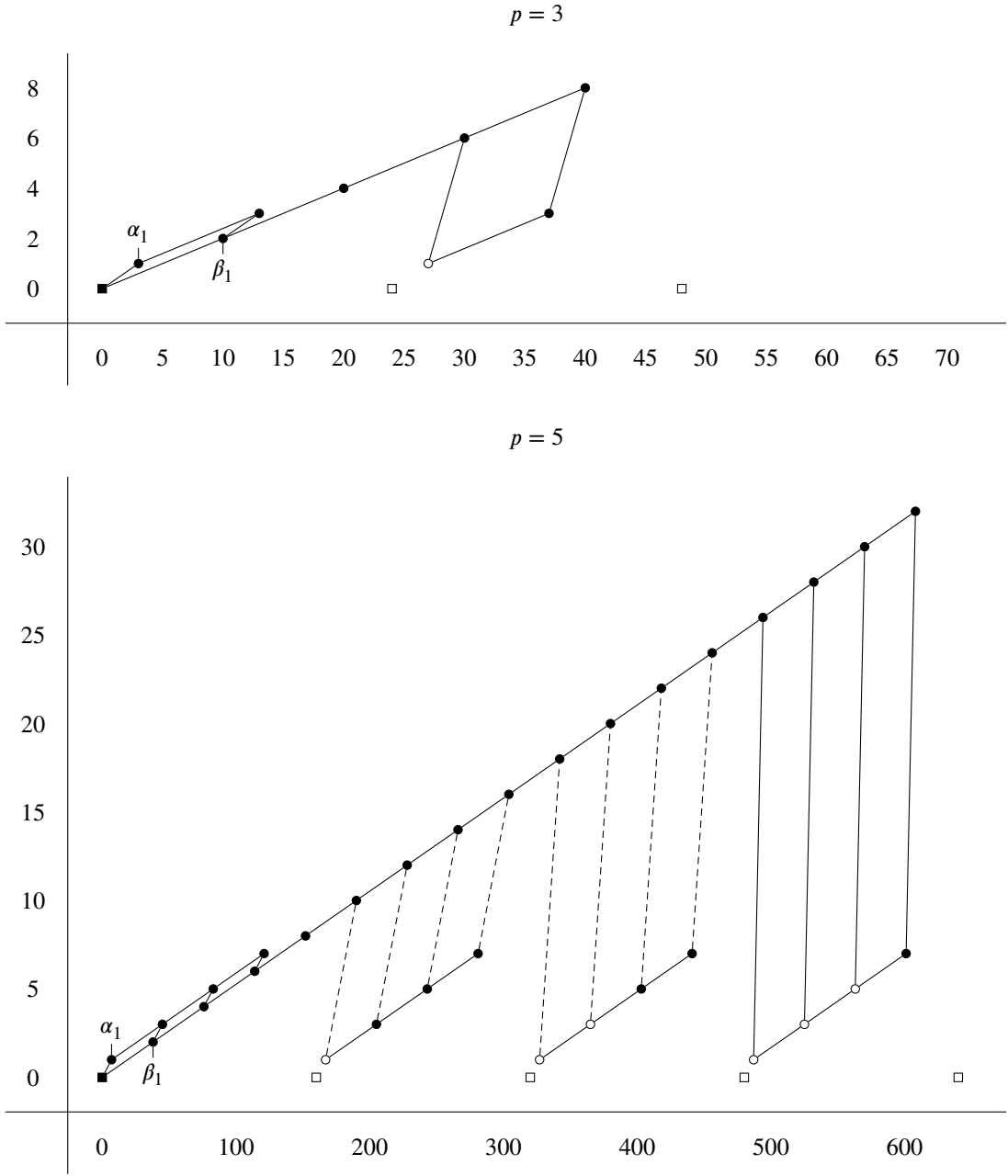

    \printpage[name=EO_3]
    \vskip13pt
    \printpage[name=EO_5]
    \caption{The homotopy of $\EO$ at the primes 3 and 5. The $y$-axis is the homotopy fixed point filtration. Most classes in filtration 0 are omitted.  The lines indicate $\alpha$ and $\beta$ multiplications, the dashes lines when $p=5$ indicate Toda brackets $\toda{\alpha_1,\alpha_1,-}$ or $\toda{\alpha_1,\alpha_1,\alpha_1,-}$. The periodicity for $p=3$ is $72$ and for $p=5$ is $800$.  The Hurewicz image classes are solid, the remaining classes are open.}
    \label{fig:homotopy-of-EO}
    \end{figure}
    \clearpage
}
\makeatletter
\immediate\write\@auxout{\unexpanded{
    \newlabel{fig:homotopy-of-EO}{{1}{3}{The homotopy of $\EO $ at the primes 3 and 5. The $y$-axis is the homotopy fixed point filtration. Most classes in filtration 0 are omitted. The lines indicate $\alpha $ and $\beta $ multiplications, the dashes lines when $p=5$ indicate Toda brackets $\toda {\alpha _1,\alpha _1,-}$ or $\toda {\alpha _1,\alpha _1,\alpha _1,-}$. The periodicity for $p=3$ is $72$ and for $p=5$ is $800$. The Hurewicz image classes are solid, the remaining classes are open}{figure.1}{}}
}}
\immediate\write\@auxout{\unexpanded{\newlabel{fig:homotopy-of-EO@cref}{{[figure][1][]1}{[1][3][]3}}}}
\makeatother

Hopkins and Miller computed the homotopy fixed point spectral sequence $H^*_{G}(E_*)\Rightarrow \EO_*$ up to some permanent cycles on the zero line. The homotopy of $\EO_*$ for $p=3$ and $p=5$ is illustrated in \Cref{fig:homotopy-of-EO}. We review the facts we need about this spectral sequence in \Cref{subsec:review-of-EO*}. A more detailed description appears in section 2 of \cite{nave}.

Let $\alpha_1\in \pi_{2p-3}(S^0)$ be the first nontrivial element of $p$-primary stable homotopy. The Toda bracket of $\alpha_1$ with itself $p$ times is  \[\toda{\underbrace{\alpha_1,\ldots,\alpha_1}_p}=\beta_1.\]
This Toda bracket is the obstruction to building a $(p+1)$-cell complex with a single cell in dimension $2k(p-1)$ for $k\in\{0,\ldots,p\}$ where all attaching maps are given by $\alpha_1$.
The Toda brackets $\toda{\alpha_1,\ldots,\alpha_1}$ of length $l-1<p$ vanish so there is an $l$-cell complex with a cell in each dimension $k(p-1)$ where $k\in\{0,\ldots,l-1\}$ and attaching maps $\alpha_1$ when $1\leq l\leq p$.
Call this complex $X_{l}$. The complex $X_p$ is central to the study of $\EO$ theory because $\EO\sm X_p$ has a natural complex orientable ring spectrum structure (\Cref{EO-sm-Xp-cx-orientable}).
We show in \Cref{X_l-unique} that $X_l$ is uniquely determined by its $\HFp$ homology.

\subsection*{Results about Orientations}
\label{sec-orientations}
Say that a spectrum is $k$-sparse if it only has cells in dimensions in a single congruence class modulo $k$. In this section we apply our results to show that certain complex vector bundles are $\EO$-orientable. We are working at an odd prime so the $p$-local map $\mathit{BO}\to \BU$ is a retract and all of these results apply equally well to real vector bundles. The only fact from the rest of the paper used here is the following mild generalization of \Cref{4p-4_orientation}:
\restate{sparse-connective-implies-orientable}

The space $\OS{\MU}{4p-4}^{hC_{p-1}}$ is $(2p-2)$-sparse and $2p$-connective so \Cref{sparse-connective-implies-orientable} implies that any map $\OS{\MU}{4p-4}^{hC_{p-1}}\to \BGLS$ is $\EO$-orientable. The space $\OS{\MU}{4p-4}^{hC_{p-1}}$ occurs as the Adams summand of $\OS{\MU}{2p}$. We will now use the Adams conjecture to deduce that the standard vector bundle on $\OS{\MU}{2p}$ is $\EO$-orientable.

\begin{thm}[Adams Conjecture]
Let $l\in \Z_p$ be a primitive $(p-1)$st root of unity. Let $\psi^l$ be the corresponding Adams operation. The composite
\[\begin{tikzcd}
\BU \rar["\psi^l"] & \BU \rar["J"] & \BGLS
\end{tikzcd}\]
is null.
\end{thm}
Since $\psi^l$ acts as an equivalence on all of the summands of $\BU$ other than the Adams summand $Y_{2p-2}=\OS{\BPo}{2p-2}$, we get the following form of the Adams conjecture which is how we will apply it:
\begin{cor}
\label{only-Y2p-2-matters}
A map $X\to \BU\to \BGLEO$ is null if and only if the map $X\to \BU \to Y_{2p-2}\to \BU \to \BGLEO$ is null.
\end{cor}
The complex orientation map $\MU \to \ku$ gives a map $\OS{\MU}{2p}\to \OS{\ku}{2p}$ and $\beta^{p-1}$ is a map $\OS{\ku}{2p}\to \OS{\ku}{2} = \BU$. Composing these gives us a standard map $\OS{\MU}{2p}\to \BU$.
\begin{thm}
\label{MU2p-orientation}
Let $f$ be the standard map $\OS{\MU}{2p}\to \BU$. There is a unital map $\Mf\to \EO$.
\end{thm}

\begin{proof}
By Wilson's thesis \cite{wilson}, there is a splitting $\OS{\MU}{2p}\simeq \prod_{i} \susp^{2(p-1)s_i} Y_{2k_i}$. Let $A=\prod_{k_i\not\equiv 0\pmod{p-1}} \susp^{2(p-1)s_i} Y_{2k_i}$ and $B=\prod_{k_i\equiv 0\pmod{p-1}} \susp^{2(p-1)s_i} Y_{2k_i}$ so that $\OS{\MU}{2p}\simeq A\times B$. The map $B\to \OS{\MU}{2p}\to \BU\to Y_{2p-2}$ is null. The map $A\to Y_{2p-2}$ factors through $Y_{4p-4}$, so that the composite $A\to\BGLEO$ is null by \Cref{4p-4_orientation}.
\end{proof}

\begin{cor}
\label{tensor-p-orientable}
Let $V_1,\ldots, V_p\colon Z\to \BU$ be $p$ virtual dimension zero complex vector bundles on a space $Z$. Let $V=\bigotimes_{i=1}^p V_i$. The structure map $V\colon Z\to \BU$ factors through $M\OS{\MU}{2p}$ and so $V$ is $\EO$-orientable.
\end{cor}

\begin{proof}[Proof]
Let $\theta\colon \MU\to \ku$ be the complex orientation. This gives a map $\OS{\MU}{2}\to \OS{\ku}{2}$. The $\MU$ Chern class $c_1^{\MU}\in \MU^2\BU$ corresponds to a map $c_1^{\MU}\in [\OS{\ku}{2},\OS{\MU}{2}]$. By naturality of Chern classes, $\theta(c_1^{\MU})=c_1^{\ku}\in [\OS{\ku}{2},\OS{\ku}{2}]$, which is the identity map. Thus, $c_1^{\MU}$ is a section of $\theta$:
\[\begin{tikzcd}
\OS{\MU}{2}\rar[yshift=-2pt,"\theta"'] &[10pt] \lar[yshift=2pt,"c_1^{\mathrlap{\MU}}"']\OS{\ku}{2}
\end{tikzcd}\]
Given a vector bundle $V_i\in [Z,\OS{\ku}{2}]$ we get an element $c_1^{\MU}(V_i)\in [X,\OS{\MU}{2}]$. Multiplying these together gives $\Pi^{\MU}=\prod c_1^{\MU}(V_i)\in [Z,\OS{\MU}{2p}]$. This gives a factorization of the structure map $V\colon Z\to \BU$ through $\OS{\MU}{2p}$ and by \Cref{MU2p-orientation}, $V$ is $\EO$-orientable.
\end{proof}

\begin{cor}
\label{pV-EO-orientable}
Let $V\colon Z\to \BU\times \Z$. Then $pV$ is $\EO$-orientable.
\end{cor}

\begin{proof}
\def\sqtimes{\mathop{\scalerel*{\boxtimes}{\otimes}}}
It suffices to check this on the universal example $\BU\times\Z = \prod_{i=0}^{p-2} \OS{\BPo}{2i}$. The spaces $\OS{\BPo}{2i}$ are all even so there is a Kunneth isomorphism $\bigotimes_{i=0}^{p-2}\KU^0(\OS{\BPo}{2i}) \cong \KU^0(\BU)$ where the map sends a collection of bundles $V_{0},\ldots, V_{p-2}$ to their external tensor product $V_0\sqtimes \cdots\sqtimes V_{p-2}$. Thus $pV = p(V_0\sqtimes \cdots\sqtimes V_{p-2}) = (pV_0)\sqtimes V_{1}\sqtimes\cdots\sqtimes V_{p-2}$. To check that the external tensor product is orientable, it suffices to show that each of the bundles is individually orientable. For $i\neq 0$, the composite $\OS{\BPo}{2i}\to \BU \to Y_{2p-2}\times \Z$ is null so the bundles $V_i$ are spherically orientable. The remaining case we need to check is that $pV_{0}$ is $\EO$-orientable.

The space $Y_{2p-2}\times \Z$ is $(2p-2)$-sparse so $\EO$ orientations of bundles over $Y_{2p-2}\times \Z$ are Chern determined. To show that $pV_0$ is $\EO$-orientable, we need to check that $\psi_{p-1}(pV_0)$ is divisible by $p$. Power sum polynomials are additive, so $\psi_{p-1}(pV_0)=p\psi_{p-1}(V_0)$.
\end{proof}

\begin{cor}
Let $V\colon Z\to \BU\times \Z$. Then $V^{\otimes p}$ is $\EO$-orientable.
\end{cor}

\begin{proof}
\def\Vbar{\overline{V}}
Let $d=\dim(V)$ and $\Vbar=V-d$. By \Cref{tensor-p-orientable}, $\Vbar^{\otimes p}$ is $\EO$-orientable. Then $V^{\otimes p} = (\Vbar + d)^{\otimes p} = \Vbar^{\otimes p} + \sum_{i=1}^{p-1} \binom{p}{i} \Vbar^{\otimes i} + d^p$. Since $\binom{p}{i}$ is divisible by $p$, every term in this sum is orientable.
\end{proof}

We can similarly combine \Cref{tensor-p-orientable} and \Cref{pV-EO-orientable} to see that if $V_1$, \ldots, $V_p$ are complex vector bundles with dimension divisible by $p$ then $V_1\otimes\cdots\otimes V_p$ is $\EO$-orientable.

\subsection*{Outline}
Given an $\EO$-module $M$ we get an associated $\KCp$-module $\EEO_*(M)/\mfrak = \pi_* (E \sm_{\EO} M)/\m$. This decomposes into a sum of indecomposable $\KCp$ representations.
We are interested in showing that certain $\EO$-modules $M$ have a splitting that lifts the decomposition of $\EEO_*(M)/\mfrak$. Bousfield \cite{bousfield} showed at the prime 2 that many $\KO$-modules $M$ have such splittings. The following theorem is a much simplified special case (see also \cite[Theorem 1.1]{relatively-free}).
\begin{thm}
\label{even-KO-module-splitting}
Let $V_1$ be the trivial representation of $\Ftw[C_2]$ and let $V_2$ be the regular representation. Say that a $\KO$-module $M$ is even if $\KU_*^{\KO}(M)$ is even and free. If $M$ is an even $\KO$-module and $\KU_0^{\KO}(M)/2 \cong V_{1}^{\oplus k} \oplus V_{2}^{\oplus l}$ then $M\simeq \bigvee_{i=1}^k \susp^{s_i} \KO \vee\bigvee_{i=1}^l \KU$ where $s_i\in 2\Z/8\Z$ are appropriate shifts.
\end{thm}
Meier \cite{relatively-free} partially extended the results of Bousfield to the case of $\TMF_{(3)}$, but $\TMF_{(3)}$-modules are very messy and it is impossible to classify their behavior as completely as Bousfield classified $\KO$-modules. If $M$ is an $\EO$-module then $\EEO_*(M)/\mfrak$ is naturally a $\KCp$-module.
If we let $V_l$ be the length $l$ indecomposible $\KCp$ representation, then we have a splitting $\EEO_*(M)/\mfrak\cong \bigoplus_{l=1}^{p} V_{l}^{\oplus m_l}$. 
We show in \Cref{E*X_l/m=V_l} that $\E_*(X_l)/\m\cong V_l$ as $\KCp$-modules, so we might attempt to generalize \Cref{even-KO-module-splitting} to odd primes by saying that if $X$ is an even $\EO$-module and $\EEO_*(M)/\mfrak\simeq \bigoplus_{i\in S} V_{l_i}$ then $\EO\sm X\simeq \EO\sm \bigwsum_{i=1}^d \susp^{s_i} X_{l_i}$. 
For most spectra this is far from being true -- the case when $p=2$ works because the only odd dimensional homotopy class in $\KO_*$ is $\eta v^i$ where $v$ is the periodicity element. 
By contrast, there are plenty of odd dimensional classes in $\EO_*$. We call an $\EO$-module algebraic in the case where such a splitting holds:
\restate{defn-algebraic-EO-theory}
This is closely related to Meier's notion of a standard vector bundle, see the discussion on page \pageref{defn-algebraic-EO-theory}.

As a replacement for the evenness assumption, we consider stronger ``sparsity'' conditions on the cell structure of spectra. 
Inspired by the Adams splitting of $\CPinfty$, we consider $(2p-2)$-sparse spectra. The homotopy of $\EO_{*}$ has $p-1$ different nonzero stems in degrees $2(p-1)k-1$, but the only such stem with a nontrivial Hurewicz image is $\pi_{2p-3}$ which contains $\alpha_1$ (see \Cref{fig:homotopy-of-EO}). 
As a consequence, every $(2p-2)$-sparse connective spectrum has algebraic $\EO$ theory:
\restate{2n-sparse-algebraic-EO-theory}
\Cref*{2n-sparse-algebraic-EO-theory} applies to show that $X_i\sm X_j$ has algebraic $\EO$ theory. As a consequence, smash products of algebraic $\EO$-modules are algebraic. \Cref*{2n-sparse-algebraic-EO-theory} can also be used to show that several naturally occurring spectra have algebraic $\EO$ theory, for instance $\CPinfty$ stably splits into a sum of $p-1$ spectra which are each $(2p-2)$-sparse, so $\CPinfty$ has algebraic $\EO$ theory.

The groups $\EO_{2pk-1}$ are zero for all $k$, so we get a simpler result for $2p$-sparse spectra:
\restate{2p-sparse-algebraic-EO-theory}

We observe that $\KCp$-free summands of $E_*(Z)/\mfrak$ lift to spectrum level splittings because the $E_2$ page of the homotopy fixed point spectral sequence for $\EO\sm X_p$ is concentrated on the zero line:
\begin{reformulatetheorem}{EO-splitting-free}
If $M$ is a finite $\EO$-module and $\pi_*(E\sm_{\EO}M)/\mfrak \cong \susp^s F\oplus V$ where $F$ is a free $\KCp$-module on one generator and $V$ is some complement then $M\simeq \EO\sm \susp^s X_p \vee M'$ for some $\EO$-module $M'$ with $\EEO_*(M')=V'$.
\end{reformulatetheorem}
For many important spectra, $E_*(Z)/\mfrak$ has a large $\KCp$-free summand, so \Cref*{EO-splitting-free} is useful. Unlike the other results in this paper, \Cref*{EO-splitting-free} directly generalizes to $E_{k(p-1)}^{hC_p}$. We intend to explore the consequences of this higher height generalization in future work.

As a consequence of our splitting theory, we deduce some closure properties of the category of algebraic $\EO$-modules. It is clear from the definition that the category of algebraic $\EO$-modules is closed under sums and retracts. \Cref{algebraic-EOmod-closed-under-union} shows that algebraic $\EO$-modules are closed under ``unions''. \Cref{alg-EO-mod-closed-under-smash} says that algebraic $\EO$-modules are closed under smash products. \Cref{alg-EO-mod-closed-under-sym} says that algebraic $\EO$-modules are closed under $i$th symmetric powers for $i<p$. Algebraic $\EO$-modules are not closed under cofiber sequences, though if a map $M\to N$ of algebraic $\EO$-modules induces an injection or a surjection $\EEO_*(M)\to \EEO_*(N)$ then the cofiber is algebraic.

If a spectrum $X$ has algebraic $\EO$ theory, it is easy to compute the homotopy type of $\EO\sm X$. Let $P(1)^*\subseteq A^*$ be the sub Hopf algebra of the Steenrod algebra generated by $P^1$. Explicitly, $P(1)^* = \Fp[P^1]/(P^1)^p$ with $P^1$ primitive. Let $P(1)_* = \Fp[\xi_1]/(\xi_1^p)$ be the dual quotient Hopf algebra of $A_*$. If a spectrum $X$ has algebraic $\EO$ theory, the homotopy type of $\EO\sm X$ is determined by the $P(1)_*$-coaction on $\HFp_*(X)$. The indecomposable representations of $P(1)_*$ are cyclic modules of length at most $p$. Let $W_l=\HFp_*(X_l)$ be the $P(1)_*$-comodule of length $l$.
\restate{splitting-from-HFp}

We also use our determination of the $C_p$ action on $E_*(X_p)$ to prove that the map $E^{hC_p}\to E$ is Galois. This is a special case of the result due to Devinatz \cite{devinatz-homotopy-fixed-point-spectra} that $E_h^{G}\to E_h$ is Galois for any finite subgroup $G$ of any height Morava $E$-theory. See \cite[Theorem 5.4.4(b)]{rognes-galois}. We then show that for any $\EO$-module there is a strongly convergent Adams spectral sequence $H^*_G(\pi_*(E\sm_{\EO}M)) \Rightarrow \pi_*(M)$. This is also originally due to Devinatz \cite[Corollary 3.4]{devinatz-homotopy-fixed-point-spectra}. Our proof is more explicit and less technical than the proof of Devinatz but relies on having the spectrum $X_p$ as a ``witness'' to the equivalence.

In \Cref{sec-X_l-unique}, we prove that the spectra $X_l$ are determined by their $\F_p$-homology.
In \Cref{sec-E_*X_l}, we compute the $C_p$ action on $E_*(X_l)/\mfrak$. In \Cref{sec-galois}, we prove that the map $\EO\to E$ is Galois. We also show that the relative Adams spectral sequence based on $\EO\to E$ is strongly convergent for all $\EO$-modules and has $E_2$ page given by group cohomology $H^*_G(\EEO_*M)$. In \Cref{sec-splittings}, we prove a collection of technical splitting results that can be used to deduce that a spectrum is algebraic based on its $\Fp$ homology.
In \Cref{sec:Y2p}, we prove that $Y_{2p}$ has algebraic $\EO$-theory and that every sphere bundle over $Y_{2p}$ is $\EO$-orientable.
In the appendix, we present the facts about symmetric powers of $P(1)_*$-comodules that we need for \Cref{sec:Y2p}. None of the material after \Cref{subsec:orientations} is necessary to prove the results quoted in the introduction.

\subsection*{Acknowledgements}
I would like to thank the topology community as a whole; I have very much enjoyed my time as a part of it.
Thanks to Eric Peterson for being my long-time homotopy theory mentor, starting after my second year as an undergraduate. He went far out of his way to help me learn, and his perspective on homotopy theory has been very valuable to me. He is the person who introduced me to orientation theory and who first told me about the problem of orienting $\EO$ theory. Eric is also a dear friend. Thanks to Haynes Miller for recommending $\EO$ theory to me and for being a wonderful and caring doctoral advisor. Thanks to Lennart Meier for his thesis which was a critical source of ideas for this project.
Thanks to Eva Belmont and Jeremy Hahn for comments on a draft. Thanks to Robert Berkulund, Prasit Bhattacharya, Sanath Devalapukar, Mike Hill, Mike Hopkins, Achim Krause, Doug Ravenel, John Rognes, Christopher Ryba, Andy Senger, Siddharth Venkatesh, and Zhouli Xu for helpful conversations.

\section{Uniqueness of \texorpdfstring{$X_l$}{X\_l}}
We prove that the spectra $X_l$ are uniquely determined by their $\Fp$ cohomology.
\label{sec-X_l-unique}
\begin{lem}
Let $Z=\BP^{2k}$ be a skeleton of $\BP$. Suppose that $Y$ is some other finite $p$-complete spectrum such that $\HFp_*(Y)\cong\HFp_*(Z)$ as Steenrod comodules. Then $Y\simeq Z$.
\end{lem}
\begin{proof}
There is a map $Z\to \BP$ including the skeleton of $\BP$ which gives a permanent cycle $\theta$ in the Adams spectral sequence $\Ext_{A_*}^{s,t}(\Fp,\HFp_*(DZ \sm \BP))$.  Because $\HFp_*(Y)\cong \HFp_*(Z)$ there is an isomorphism of $E_2$ pages $\Ext_{A_*}^{s,t}(\Fp,\HFp_*(DZ \sm \BP))\cong \Ext_{A_*}^{s,t}(\Fp,\HFp_*(DY \sm \BP))$ using the Kunneth isomorphism. We wish to show that the element $\theta\in E_2^{0,0}\!\ASS(DY\sm \BP)$ is a permanent cycle. Because $Z$ is even, $DZ\sm \BP$ splits as a wedge of copies of $\BP$ and $E_2\!\ASS(DZ\sm \BP)\cong E_2\!\ASS(\BP)\otimes \HFp_*(DZ)$. Both $E_2\!\ASS(\BP)$ and $\HFp_*(DZ)$ are even, so $E_2\!\ASS(DZ\sm \BP)$ is even. Thus, the spectral sequence collapses at $E_2$ and $\theta$ is a permanent cycle.

We deduce that there is a map $Y\to \BP$. Since $Y$ has no homology above degree $2k$, the map $Y\to \BP$ factors through $\BP^{(2k)}=Z$. The factored map $Y\to Z$ is an isomorphism on homology so $Y\simeq Z$.
\end{proof}
A spectrum $Z$ with the cohomology of $X_l$ can be obtained as the $2(p-1)(l-1)$-skeleton of $\BP$, so as a special case we deduce:
\begin{lem}
A spectrum $Y$ is equivalent to $X_l$ if and only if $\HFp_*(Y) \cong W_l$ where $W_l=\HFp_*(X_l)$  is the Steenrod comodule $\Fp\{x_0,\ldots, x_{l-1}\}$ with $\deg{x_k}=2k(p-1)$ and Steenrod coaction given by $\Psi(x_k)=\xi_1\otimes x_{k-1}+\cdots$ for $k\geq 1$.
\label{X_l-unique}
\end{lem}
\section{The \texorpdfstring{$C_p$}{C\_p} action on \texorpdfstring{$E_*(X_l)$}{E\_*(X\_l)}}
\label{sec-E_*X_l}
\def\FMLGPS{\mathop{\mathrm{FmlGrps}}}
Set $n=p-1$ for the rest of the paper.
We begin this section with a brief review of the facts we need about the Morava stabilizer group. We then compute the $\KCp$ action on $E_*(X_l)/\mfrak$ and show that $E_*(X_p)$ is a free $E_*[C_p]$-module. We will deduce that $\EO\sm X_p \simeq \EhC$. Since $n^2$ is relatively prime to $p$, $\EhC$ is complex orientable.

Let $\FMLGPS$ be the category of pairs $(k, \Gamma)$ where $k$ is a perfect characteristic $p$ field and $\Gamma$ is a formal group over $k$. The morphisms $(k,\Gamma)\to(k',\Gamma')$ are pairs consisting of a field homomorphism $f\colon k\to k'$ and an isomorphism of formal groups $f^*\Gamma\to \Gamma'$. The Hopkins-Miller theorem says there is a functor $\FMLGPS\to E_{\infty}\text{-Rings}$ which sends a pair $(k,\Gamma)$ to the corresponding Morava $E$ theory $E(k,\Gamma)$. This implies that there is an action of the automorphism group $\Aut(\Gamma)$ on $E(k,\Gamma)$ by $E_{\infty}$ ring maps. The group $\G = \Aut(\Gamma)$ is called the Morava stabilizer group. See section 2 of \cite{bhattacharya} for a nice overview of the Morava stabilizer group.

The Morava stabilizer group of a height $n$ Morava $E$-theory contains elements of order $d$ if and only if the degree of $\Qp(\zeta_d)$ over $\Qp$ divides $n$ where $\zeta_d$ is a primitive $d$th root of unity.
In particular, $\Qp(\zeta_p)$ has degree $p-1$, so there are $p$-torsion elements in $\G$ if and only if $p-1$ divides $n$.
In this paper we study the simplest such case, when $n=p-1$. Let $E = E(\Fpn, \Gamma_n)$ where $\Gamma_n$ is the height $n$ Honda formal group over $\Fpn$ and let $\G = \Aut(\Gamma_n)$ be the corresponding Morava stabilizer group.
There is a $\G$-action on $E_*(Z) = \pi_*(L_{K(n)}E\sm Z)$ for any spectrum $Z$ by letting $\G$ act in the standard way on $E$ and trivially on $Z$.

There is an isomorphism $E_* \cong \W(\Fpn)\ps{u_1,\ldots,u_{n-1}}$. Let $\m=(p,u_1,\ldots,u_{n-1})$ be the maximal ideal of $E_*$ and let $K_* = E_*/\m = \Fpn[u^{\pm}]$. For $Z$ a torsion free spectrum, $E_*(Z)/\mfrak\cong K_*(Z)$ where $K$ is any Morava $K$-theory corresponding to $E$. Let $E_*E=\pi_*(L_{K(n)} E\sm E)$. There is an isomorphism $E_*E\cong \Homcts$ where for $g\in \G$ the evaluation map
\[\begin{tikzcd}
E_*E \rar["\cong"]  & \Homcts\rar["ev_{g}"] & E_*
\end{tikzcd}\]
is the image of the map
\[
\begin{tikzcd}
E\sm E \rar["g\sm id"] & E\sm E \rar["m"] & E
\end{tikzcd}\]
under the functor $\pi_*(L_{K(n)}(-))$.

Let $G$ be a maximal finite subgroup of $\G$ containing an element of order $p$. According to Corollary 1.30 and Theorem 1.31 of \cite{finite-subgroups}, any two such subgroups $G$ are conjugate in $\G$ and $G$ is abstractly isomorphic to the semidirect product $C_p\rtimes C_{n^2}$ where the action is given by the surjection $C_{n^2}\to C_n\cong \Aut(C_p)$. Let $\EO = E^{hG}$. For an $\EO$-module $M$ we write $\EEO_*(M) = \pi_*(E\sm_{\EO} M)$. There is an action of $G$ on $E$ by $\EO$-automorphisms so this gives a $G$ action on $\EEO_*(M)$ for any $M$. We will show in the next section that for any $\EO$-module is a relative Adams spectral sequence $H^*_G(\EEO_*(M))\Rightarrow \pi_*(M)$. Our plan is to use this Adams spectral sequence to understand $M$ so we will need to compute the $E_*[G]$-module structure on $\EEO_*(M)$. To allow explicit calculation, we compute the $K_*[G]$-module structure on $\EEO_*(M)/\mfrak$ and then use Nakayama's lemma to make the conclusions we need about $\EEO_*(M)$.

Let $\zeta\in C_p$ be a generator. There is an isomorphism $K_*[C_p]=K_*[\zeta]/(\zeta^p-1)\cong K_*[s]/(s^p)$ where the map sends $\zeta\mapsto s+1$. The coproduct is given by $\delta(s) = s\otimes 1 + 1\otimes s + s\otimes s$. Let $V_l$ be the cyclic module over $K_*[s]/(s^p)$ of length $l$.

\begin{prop}
\label{E*X_l/m=V_l}
$\E_*(X_l)/\m\cong V_l$ as $K_*[C_p]$-modules.
\end{prop}

To prove this, we are going to pass from information about the Steenrod coaction on $\HFp_*(Z)$ to information about the Morava stabilizer group action on $E_*(Z)$ through the $\BP_*\BP$-coaction on $\BP_*(Z)$ by considering the maps $\BP_*(Z)\to \HFp_*(Z)$ and $\BP_*(Z)\to E_*(Z)$.

If $Z$ is a torsion free connective spectrum then $\BP_*(Z)$ is $\BP_*$-free so $\HFp_*(Z) = \Fp\otimes_{\BP_*} \!\BP_*(Z)$ and $E_*(Z) = E_*\otimes_{\BP_{*}} \!\BP_*(Z)$. Let $\phi\colon \BP\to E$ and $\pi\colon \BP\to \HFp$ be the maps induced by the complex orientations of $E$ and $\HFp$.
\[\begin{tikzcd}
\BP \rar["\pi"]\dar["\phi"] & \HFp  &  \BP_*\BP \rar["\pi"]\dar["\phi"] & \HFp_*\HFp\\
E                          &       &  E_*E \mathrlap{{}=\Homcts}
\end{tikzcd}\]

If $\BP_*(Z) \cong \BP\{\supBP{z_i}\}_{i\in S}$, we write $\supE{z_i} = \phi(\supBP{z_i})$ and $\supFp{z_i} = \pi(\supBP{z_i})$ so then $E_*(Z) \cong E_*\{\supE{z_i}\}_{i\in S}$ and $\HFp_*(Z)\cong \F_p\{\supFp{z_i}\}_{i\in S}$. For $E$ some cohomology theory, let $\supE{I_d}(Z) = \ker(E_*(Z)\to E_*(Z_{(d)}))$ where $Z_{(d)}$ is the cofiber of the inclusion of the $(d-1)$-skeleton of $Z$.

Consider the map $\BP_*\BP\to A_*$. This sends $t_1\mapsto -\xi_1$. If $Z$ is torsion free and $\supFp{z_{k}}\in \HFp_*(Z)$ has a nontrivial $P^1_*$ action $P^1_*(\supFp{z_{k}}) = \supFp{z_{k-2n}}$ then
\begin{align*}
\Psi(\supFp{z_{k}}) &= 1\otimes \supFpAlign{z_{k}} + \xi_1\otimes\supFpAlign{z_{k-2n}} \pmod{A_*\otimes_{\Fp} \supFp{I_{k-2n}}(Z)}.\\
\intertext{In this case there are lifts $\supBP{z_k}$, $\supBP{z_{k-2n}}\in \BP_*(Z)$ and}
\Psi(\supBP{z_{k}})&= 1\otimes \supBPAlign{z_{k}} - t_1\otimes \supBPAlign{z_{k-2n}} \pmod{\BP_*\BP\otimes_{\BP_*}\supBP{I_{k-2n}}(Z)}.
\end{align*}
For $g\in\G$ and $\theta\in \BP_*\BP$ we can evaluate $\theta(g)\in E_*$ using the map $\BP_*\BP\to E_*E=\Homcts$. A strict automorphism $g$ of a $p$-typical formal group $\Gamma$ corresponds to a certain power series, namely $g(s)=s+^{\Gamma}\sum^{\Gamma} a_i s^{p^i} \in E_*\ps{s}$. Then $t_i(g)=a_i$.

\begin{lem}
\label{stabilizer-action-coaction}
If $g\in \G$ and $\supBP{z_{k}}\in \BP_{k}(Z)$ has coaction $\Psi(\supBP{z_{k}}) = \sum \theta_i\otimes \supBP{z_{i}}$ where $\theta_i\in \BP_*\BP$, then $g_*(\supE{z_{k}})=\sum_{i} \theta_i(g)\supE{z_i}$.
\end{lem}

\begin{proof}
Recall that $E_*E \cong \Hom(\G, E_*)$ where for each $g\in \G$ there is a commutative diagram:
\[\begin{tikzcd}
\pi_*(L_{K(n)}(E\sm E)) \dar["\pi_*(g\sm id_E)"'] \rar[equal] & \Hom(\G,E_*)\dar["ev_{g}"]\\
\pi_*(L_{K(n)}(E\sm E)) \rar["\pi_*(m)"] & E_*
\end{tikzcd}\]

Consider the maps
\[\begin{tikzcd}
E\sm Z \rar["\Psi",yshift=0.2em] & \lar["m\sm id_{Z}",yshift=-0.2em] E\sm E \sm Z
\end{tikzcd}\]
where $\Psi\colon E\sm Z = E\sm S^0\sm Z \to E\sm E\sm Z$ is the unit map in the middle. If we let $\G$ act by the standard action on the leftmost $E$ factor and trivially on the other factors, these maps are $\G$-equivariant. Because the left unit map $E_*\to E_*E$ is flat, there is an isomorphism $\pi_*(E\sm E\sm Z)\cong E_*E\otimes_{E_*} E_*(Z)$. Thus, the action of $g$ on $E_*(Z)$ factors as:
\[\begin{tikzcd}
E_*(Z)\dar["\Psi"'] \rar["g_*"] & E_*(Z)\\
E_*E \otimes_{E_*}E_*(Z)\rar["g_*\otimes 1"]& E_*E\otimes_{E_*}E_*(Z)\uar["m_*"']
\end{tikzcd}\]
Since $m_*\circ(g_*\otimes 1) = ev_{g}$ under the isomorphism $E_*E\cong \Hom(\G,E_*)$, we see that $g_*(z) = (ev_{g} \otimes 1) \circ \Psi(z)$.

We have a commutative diagram:
\[\begin{tikzcd}
\BP_*(Z)\rar["\Psi"]\dar &  \BP_*\BP \otimes_{\BP_*} \BP_*(Z)\dar\\
E_*(Z)\rar["\Psi"]\drar["g_*"'] &  E_*E \otimes_{\BP_*} E_*(Z)\dar["ev_{g}\otimes 1"]\\
                              & E_*(Z)
\end{tikzcd}\]
We deduce that $g_*(\supE{z}) = (ev_{g}\otimes \phi)(\Psi(\supBP{z}))$ as desired.
\end{proof}

Recall that $\zeta\in C_p$ is a generator. Let $v = t_1(\zeta) \in E_{2n}$. It is well known that $v$ is a unit (see for instance \cite[bottom of page 438]{odd-primary-arf}). Specializing \Cref{stabilizer-action-coaction} to the case we care about, if
\begin{align*}
\Psi(\supBP{z_{k}}) &= 1\otimes \supBP{z_{k}} + t_1\otimes \supBP{z_{k-2n}} \pmod{\BP_*\BP\otimes_{\BP_*} \supBP{I_{k-2n}}(Z)}\\
\intertext{then}
\zeta_*(\supE{z_{k}}) &= \phantom{1\otimes{}} \supCohThyAlign{E}{\BP}{z_{k}} + \phantom{{}_{1}\otimes\,} v\supE{z_{k-2n}}  \phantom{\BP_*\BP\otimes_{\BP_*}{}}\pmod{\supE{I_{k-2n}}(Z)}.
\end{align*}

\begin{proof}[Proof of \Cref{E*X_l/m=V_l}]
The spectrum $X_l$ is torsion free so the above discussion applies. Recall that
\[\HFp_*X_p\cong \Fp\{x_0,\ldots,x_{p-1}\}\]
where $|x_k| = 2kn$. For $0\leq k<p-1$,
\begin{align*}
\Psi(\supFp{x_{k}}) &= 1\otimes \supFpAlign{x_k} + \xi_1\otimes \supFpAlign{x_{k-1}} \pmod{A_*\otimes_{\Fp} \supFp{I}_{2(k-1)n}(Z)}\\
\intertext{This implies that in $\BP_*X_p$,}
\Psi(\supBP{x_{k}}) &= 1\otimes \supBPAlign{x_{k}} - t_1\otimes \supBPAlign{x_{k-1}} \pmod{\BP_*\BP\otimes_{\BP_*} I^{\BP}_{2(k-1)n}(Z)}\\
\intertext{and the action of $\zeta$ on $E_*(X_p)$ is given by}
\zeta_*(\supE{x_k})        &= \phantom{1\otimes{}}\supEAlign{x_{k}}         -   \phantom{{t}_{1}\otimes{}\!}\supEAlign{x_{k-1}} \pmod{I_{2(k-1)n}^{E}(Z)}.
\end{align*}
In matrix form when $p = l = 5$ this looks like:
\[\begin{pmatrix}
1 & v & * & * & * \\
0 & 1 & v & * & * \\
0 & 0 & 1 & v & * \\
0 & 0 & 0 & 1 & v \\
0 & 0 & 0 & 0 & 1 \\
\end{pmatrix}\]
This matrix is conjugate to a length $l$ Jordan block, so $E_*(X_l)/\m$ does not split and hence it is a length $l$ indecomposable $\KCp$ representation.
\end{proof}

In particular, $\zeta$ acts trivially on $K_* = K_*(X_1) = V_1$ and $K_*(X_p) = V_p$ is a free $K_*[C_p]$-module. By Nakayama's lemma, we deduce that $E_*(X_p)$ is a free $E_*[C_p]$-module.
\begin{cor}
$\E_*(X_p)$ is a free $E_*[C_p]$-module.
\label{EXp-free}
\end{cor}

\begin{lem}
If $M$ is a finite $\EO$-module then $M \simeq (E\sm_{\EO} M)^{hG}$.
\end{lem}

\begin{proof}
Let $G$ act on $E\sm_{\EO} M$ by $E$ automorphisms over $\EO$. There is a natural equivariant map $M = \EO\sm_{\EO} M \to E\sm_{\EO} M$ where $G$ acts trivially on $M$, so we get a natural transformation $M \to (E\sm_{\EO}M)^{hG}$. When $M=\EO$ this is an equivalence by definition. The functor $M\mapsto \pi_*(E\sm_{\EO}M)^{hG}$ is exact, so it follows that this natural transformation is an equivalence on all finite $\EO$-modules.
\end{proof}

\begin{cor}
\label{EO-sm-Xp-cx-orientable}
$\EO\sm X_p\simeq \EhC$
\end{cor}
Since $n^2$ is relatively prime to $p$, $\EhC$ is complex orientable.
\begin{proof}
$\EO\sm X_p\simeq (E\sm X_p)^{hG}$. Now $E\sm X_p\simeq \bigvee_{p}E$ and since $\E_*(X_p)$ is a free $E_*[C_p]$-module, this equivalence can be chosen to be $C_p$ equivariant, where the action of $C_p$ on $\bigvee_{p}E$ is given by permuting the $p$ factors. It follows that
\[(E\sm X_p)^{hC_p}\simeq \left(\bigvee_{p}E\right)^{\!\!hC_p}\simeq E\]
and so
\[\EO\sm X_p \simeq \left(E\sm X_p\right)^{hG} \simeq \left(\left(E\sm X_p\right)^{hC_p}\right)^{\!\hCn}\simeq \EhC.\qedhere\]
\end{proof}

\section{The map \texorpdfstring{$\EO\to E$}{EO --> E} is Galois and the \texorpdfstring{$E$}{E}-based Adams spectral sequence for \texorpdfstring{$\EO$}{EO}-modules}
\label{sec-galois}
Here we present a proof that the maps $\EO\to \EhC$ is a Galois extension.
This is a special case of \cite[Theorem 5.4.4(b)]{rognes-galois} which Rognes attributes to Devinatz \cite{devinatz-homotopy-fixed-point-spectra}.
We wanted to prove that $EO\to E$ is Galois, but failed to do so. We cite Devinatz for this.
We then conclude that the $E$-based Adams spectral sequence is strongly convergent for $\EO$-modules and has $E_2$ page given by group cohomology $H^*_{C_p}(\EEO_*(M))\Rightarrow \pi_*(M)$.
The $E_2$ page and convergence of this spectral sequence are also due to Devinatz \cite[Corollary 3.4]{devinatz-homotopy-fixed-point-spectra}.
Recall that $n=p-1$.

\begin{defn}[Rognes {\cite[Definition 4.1.3]{rognes-galois}}]
A map $R\to S$ of $E_{\infty}$ ring spectra is an $E$-local $G$-Galois extension for a discrete group $G$ if:
\begin{enumerate}
\item $G$ acts on $S$ via $R$-algebra maps.
\item The natural map $i\colon R\to S^{hG}$ is an $E$-equivalence.
\item The map $h\colon S\sm_{R} S\to F(G_+, S)$ adjoint to
\[\begin{tikzcd}[column sep = large]
G_+\sm S\sm_{R}S\rar["\textup{act}\sm id"] & S\sm_{R}S \rar["\textup{mult}"] & S
\end{tikzcd}\]
is an $E$-equivalence.
\end{enumerate}
\end{defn}
If we let $G$ act on the left $S$ factor on $S \sm_{R} S$ and by precomposition on $F(G_+,S)$, the map $h$ is an $S[G]$-algebra map. If $h$ is an equivalence of spectra, it is automatically also an equivalence of $S[G]$-modules.

\begin{defn}[{\cite[Definition 4.3.1]{rognes-galois}}]
Let $R$ be an $E_{\infty}$ ring spectrum. An $R$-module $N$ is \emph{faithful} if any $R$-module $M$ such that $N\sm_{R}M\simeq 0$ is already zero. A map $R\to S$ of $E_{\infty}$ rings is \emph{faithful} if $S$ is faithful as an $R$-module.
\end{defn}

\def\Ccal{\mathcal{C}}
\def\Ncal{\mathcal{N}}
\begin{defn}[{\cite[Definition 3.7]{bousfield-localization}}]
Let $R\to S$ be a map of homotopy associative ring spectra. The category of $S$-nilpotent $R$ modules is the smallest subcategory $\Ncal$ of $R$-modules such that
\begin{enumerate}
\item $S\in \Ncal$
\item If $X\in \Ncal$ and $Y$ is a spectrum then $X\sm Y\in\Ncal$.
\item If $X\to Y\to Z$ is a cofiber sequence in $R$-modules and two of $X$, $Y$, and $Z$ are in $\Ncal$ then so is the third.
\item If $X\in\Ncal$ and $Y$ is a retract of $X$ then $Y\in\Ncal$.
\end{enumerate}
$R$ is $S$-nilpotent if $R$ is an $S$-nilpotent $R$-module.
\end{defn}

\begin{lem}
Let $R\to S$ be a map of homotopy associative ring spectra and suppose that $f\colon \susp^d R\to R$ is a nilpotent self map of $R$. Then $C(f)\sm_{R} M$ is $S$-nilpotent if and only if $M$ is.
\end{lem}

\begin{proof}
If $M$ is $S$-nilpotent, then $\susp^d M\to M \to C(f)\sm_{R} M$ is a cofiber sequence, and since both $M$ and $\susp^d M$ are $S$-nilpotent, so is $C(f)\sm_{R} M$. Conversely, suppose that $C(f)\sm_{R} M$ is $S$-nilpotent. We show by induction that $C(f^i)\sm_{R}M$ is $S$-nilpotent for all $i$ by induction. Suppose that $C(f^j)\sm_{R}M$ is $S$-nilpotent for $j\leq i$. The octahedral axiom gives us the following diagram, where the straight lines are all cofiber sequences:
\[\begin{tikzcd}[column sep = small, row sep = small]
&&&C(f^i)\sm_{R}M\\[6pt]
\susp^{(i+1)d}M\drar["\susp^{id}f"'] \ar[rr,"f^{(i+1)}"] && M \urar  \rar & C(f^{i+1})\sm_{R}M\uar \\[-4pt]
  & \susp^{id}M\urar["f^{i}"'] \ar[drr] &&\\[14pt]
  &&& \susp^{id}C(f)\sm_{R}M\ar[uu]
\end{tikzcd}\]
Since $C(f^i)\sm_{R}M$ and $C(f)\sm_{R}M$ are $S$-nilpotent, $C(f^{i+1})\sm_{R}M$ is $S$-nilpotent too. Because $f$ is nilpotent, $f^i$ is null for large enough $i$. Thus, $M$ is a retract of an $S$-nilpotent spectrum $C(f^{i})\sm_{R}M\simeq M\wsum M$ and so $M$ is $S$-nilpotent.
\end{proof}

\begin{lem}
Let $R\to S$ be a map of $E_{\infty}$ ring spectra and suppose that $R$ is $S$-nilpotent. Then the map $R\to S$ is faithful.
\end{lem}

\begin{proof}
Let $M$ be an $R$-module such that $S\sm_{R} M \simeq 0$. Let $\Ccal$ be the category of $R$-modules $N$ such that $N\sm_{R} M\simeq 0$. $\Ccal$ is closed under retracts because if $N'$ is a retract of $N$ then $N'\sm_{R} M$ is a retract of $N\sm_{R}M$ and retracts of zero are zero. $\Ccal$ is closed under cofiber sequences because if $N_1\to N_{2} \to N_3$ is a cofiber sequence and $N_1, N_2\in \Ccal$, then the cofiber sequence $N_1\sm_{R} M \to N_2\sm_{R}M\to N_3\sm_{R} M$ shows that $N_3\sm_{R}M\in \Ccal$. Lastly, if $N\in \Ccal$ then $(N\sm_{R} N') \sm_{R} M \simeq 0$ so $N\sm_{R} N' \in \Ccal$. This implies that $\Ccal$ contains the category of $S$-nilpotent $R$-modules, so $R\in \Ccal$ and $M = R\sm_{R} M \simeq 0$.
\end{proof}

\begin{prop}
\label{EO-to-E-faithful}
$\EO$ is $E$-nilpotent and $\EhC$-nilpotent. As a consequence, the maps $\EO\to \EhC$ and $\EO\to E$ are faithful.
\end{prop}

This is a special case of \cite[Theorem 3.3]{devinatz-homotopy-fixed-point-spectra}. Compare \cite[Proposition 5.4.5]{rognes-galois}. In order to prove this we need the following lemma:
\begin{lem}
There is a cofiber sequence $\susp^{2n-1}X_p \to C(\beta_1) \to X_p$.
\label{cofib-beta-X_p}
\end{lem}

\begin{proof}
\def\alphatwee{\widetilde{\alpha}_1}
\def\alphabar{\overline{\alpha}_1}
Let $F$ be the fiber of the inclusion of the bottom cell $S^0 \to X_p$. There is a homology isomorphism $\HFp_*F\cong \HFp_*\susp^{2n-1}X_{p-1}$ so by \Cref{X_l-unique} we deduce that $F\simeq \susp^{2n}X_{p-1}$.

Let $\alphatwee\colon S^{2n^2-1}\to X_{p-1}$ be the attaching map for $X_p$ and let $\alphabar\colon \susp^{2n-1}X_{p-1}\to S^0$ be the fiber of the the inclusion of the bottom cell $S^0\to X_p$. The composition $\alphabar\circ (\susp^{2n-1}\alphatwee)$ is the Toda bracket $\toda{\alpha_1,\ldots,\alpha_1} = \beta$. The octahedral axiom gives us the following diagram, where the straight lines are all cofiber sequences:
\[\begin{tikzcd}[column sep = small, row sep = small]
&&&X_p\\[-3pt]
S^{2pn - 2}\drar["\susp^{2n-1}\alphatwee"'] \ar[rr,"\beta"] && S^{0} \urar  \rar & C(\beta_1)\uar \\[-2pt]
  & \susp^{2n-1}X_{p-1}\urar["\alphabar"'] \ar[drr] &&\\[12pt]
  &&& \susp^{2n-1}X_p\ar[uu]
\end{tikzcd}\]
\end{proof}

\begin{proof}[Proof of \Cref{EO-to-E-faithful}]
$\EhC$ is a retract of $E$ so $\EhC$ is $E$-nilpotent. \Cref{cofib-beta-X_p} says there is a cofiber sequence:
\[\susp^{2n-1}X_p \to C(\beta_1) \to X_p.\]
Smashing this with $\EO$ gives a cofiber sequence
\[\susp^{2n-1}\EhC \to \EO\sm C(\beta_1)\to \EhC\]
so that $\EO\sm C(\beta_1)$ is $E$-nilpotent. Since $\beta_1$ is nilpotent, $\EO$ is $E$-nilpotent too.
\end{proof}

\begin{thm}
\label{EO-EhC-Galois}
The map $\EO\to \EhC_{n}$ is a faithful $C_p$-Galois extension.
\end{thm}

To prove this, we need the following lemma:
\begin{lem}
Let $k$ be a field of characteristic $p$, let $\tau\in k[C_p]$ be the trace element $\sum_{g\in C_p} g$, and let $f$ be a vector space map $k[C_p]\to k$. Then the map $k[C_p]\to \prod_{C_p} k$ adjoint to the map $C_p\times k[C_p]\to k$ given by $(g,v)\mapsto f(gv)$ is an isomorphism if and only if $f(\tau)\neq 0$.
\end{lem}

\begin{proof}
If $V$ is a $d$-dimensional $k$-vector space, a collection of $d$ maps $f_i\colon V\to k$ have product an isomorphism $V\to \prod k$ if and only if the $f_i$ generate $V^*$, so it suffices to check that $f$ generates $(k[C_p])^{\vee}$ as a $C_p$-representation. Let $\zeta\in C_p$ be a generator. Because $f(\tau)\neq 0$ and $(\zeta-1)\tau=0$ we deduce that $f$ is not in $(\zeta-1)k[C_p]^{\vee}$. However, $(\zeta-1)(k[C_p])^{\vee}$ is the unique maximal subrepresentation of $(k[C_p])^{\vee}$, so $f$ lies in no proper subrepresentation of $(k[C_p])^{\vee}$ and $f$ generates $(k[C_p])^{\vee}$ as a representation.
\end{proof}

\begin{proof}[Proof of \Cref{EO-EhC-Galois}]
Let $R=\EhC$, let $\n$ be the maximal ideal of $R_*$ and let $L_* = R_*/\n$.
We defined $\EO\to E$ as the inclusion of the $G$ fixed points, so $C_p$ acts on $R$ by $\EO$-algebra maps and
$i\colon \EO\to R^{hC_p}\simeq \EhG$ is an equivalence. So conditions (1) and (2) are satisfied. The map $\EO\to R$ is faithful by \Cref{EO-to-E-faithful}.

It remains to check condition (3). It suffices to show that $h$ is an isomorphism after taking homotopy. By Nakayama's lemma we can check that $h$ is a surjection by checking that it is a surjection after quotienting by the maximal ideal of $R$. Since $\pi_*(R\sm_{\EO} R)$ and $\pi_*\left(\prod_{C_p} R\right)$ are free $R_*$-modules of the same dimension, it will follow that $h$ is an equivalence.

The map $h\colon R\sm_{\EO} R \to \prod_{C_p} R$ has $g$ component given by the composite
\[\begin{tikzcd}
R \sm_{\EO} R \rar["g\sm id"]& R\sm_{\EO} R \rar["m"] & R
\end{tikzcd}\]
so we need to show that the sum of the $g$-conjugates of $m\colon L_*\otimes_{\eta_{L}}\pi_*(R\sm_{\EO}R)\to  L_*$ is an isomorphism.
Since $R\simeq \EO\sm X_p$, we have an equivalence of left $R$-modules $R\sm_{\EO} R \simeq R \sm X_p$. If we let $C_p$ act trivially on $X_p$, this isomorphism is $C_p$ equivariant. Consider the following diagram:
\[\begin{tikzcd}
R \rar["\eta_L"']\ar[rr,bend left=20,"id"]\drar & R\sm_{\EO} R\dar \rar["m"'] & R\\
                                                     & R\sm X_p \urar
\end{tikzcd}\]
All maps are $R$-module maps where $R$ acts on $R\sm_{\EO} R$ on the left. The map $\eta_L$ is $C_p$-equivariant, but $m$ is not equivariant for the action of $C_p$ on the left factor. Now taking homotopy and quotienting by $\n$ gives:
\[\begin{tikzcd}
L_* \rar["e"'] \ar[rr,bend left=20,"id"]& L_*[C_p] \rar["m"'] & L_*
\end{tikzcd}\]
where all maps are of $L_*$-modules and $e$ is $C_p$ equivariant. Let $\tau = \sum_{g\in C_p} g$ be the trace element. Since $e$ is an equivariant map from the trivial representation, it must be some nonzero multiple of the map $1\mapsto \tau$. We deduce that $m(\tau)$ is a unit.  By the lemma, we are done.
\end{proof}

We wish the following were a corollary:
\begin{notcor}[Devinatz \cite{devinatz-homotopy-fixed-point-spectra}]
\label{EO-to-E-Galois}
The map $EO\to E$ is a faithful Galois extension.
\end{notcor}

\begin{prop}
\label{HFPSS}
For any $\EO$-module $M$ there is a spectral sequence
\[ \HFPSS(M)\colon H^*_G(\EEO_*(M))\Rightarrow \pi_*(M)\]
and for any connective spectrum $X$ there is a map $\ANSS(X)\to \HFPSS(\EO\sm X)$.
\end{prop}

The convergence is \cite[Theorem 3.3]{devinatz-homotopy-fixed-point-spectra} and the identification of the $E_2$ term is \cite[Theorem 3.1]{devinatz-homotopy-fixed-point-spectra}. The map of spectral sequences is explained in section 11.3.3 on page 109 of \cite{kervaire}.

\begin{proof}
The left unit map $E_*\to \EEO_*E$ is flat, so given an $\EO$-module $M$ there is an $E$-based Adams spectral sequence \cite[Theorem 2.1]{baker-lazarev}
\[\Ext_{E_*^{\EO}E}(\E_*,\EEO_*(M))\Rightarrow \pi_*\left(\widehat{L}_{E}^{\EO} M\right).\]
A map of $\EO$-modules $M\to N$ is an $E$-equivalence if $E\sm_{\EO}M \to E\sm_{\EO} N$ is an equivalence. By \Cref{EO-to-E-faithful}, this is true if and only if $M\to N$ is itself an equivalence, so for any $\EO$-module, $L^{\EO}_E M\simeq M$.
\Cref{EO-to-E-faithful} implies that $\EO$ is $E$-nilpotent and the canonical maps $I\!d\to L_{E}^{\EO}\to \widehat{L}_{E}^{\EO}$ are equivalences.
By \Cref{EO-to-E-Galois} the $\Ext$ group that determines the $E_2$ page of the spectral sequence is group cohomology, so we can rewrite the $E_2$ page as:
\[H^*_G(E_*^{\EO}(M))\Rightarrow \pi_*(M).\]

The map $\BP\to E$ induces a map from the Adams Novikov spectral sequence to the $E$-based Adams spectral sequence. The $E$-based Adams spectral sequence for a spectrum $X$ corresponds to a cosimplicial object with $i$th term $E^{\sm (i+1)} \sm X$ where the face maps are unit maps and the degeneracy maps are multiplication. The map $X\to \EO\sm X$ induces a map of cosimplicial objects $E^{\sm (i+1)}\sm X\to E^{\sm_{\EO}(i+1)} \sm_{\EO} (\EO\sm X)$ where $E^{\sm_{\EO}(i+1)} \sm_{\EO} (\EO\sm X)$ corresponds to the $\EO$-based Adams spectral sequence for $\EO\sm X$. Thus, there is a corresponding map of spectral sequences $\ANSS(X)\to \HFPSS(\EO\sm X)$.
\end{proof}
There is a particularly convenient description of a minimal Adams resolution:
\begin{prop}
Any $\EO$-module $M$ has an $E$-based Adams resolution:
\[\begin{tikzcd}
M \rar & M\sm X_p \rar& \susp^{\deg{\alpha}}M\sm X_p \rar & \susp^{\deg{\beta}}M\sm X_p \rar &\susp^{\deg{\alpha}+\deg{\beta}}M\sm X_p \rar &\cdots
\end{tikzcd}\]
\end{prop}

\section{Splittings}
Recall that $n=p-1$.
\label{sec-splittings}
\begin{restatethis}{defn}{defn-cellular-EO-module}
A \emph{cellular} $\EO$-module is an $\EO$-module $M$ equipped with an \emph{Atiyah-Hirzebruch filtration} $M_0 \to M_1\to \cdots \to M$ with $M=\hocolim M_i$, such that $M_0 = \bigwsum_{j\in S_0}\susp^{s_j} \EO$ and there are cofiber sequences $\bigwsum_{j\in S_i} \susp^{s_j} \EO\to M_i \to M_{i+1}$. A cellular $\EO$-module is \emph{$k$-sparse} for $k$ a divisor of $2p^2n^2$ if all of the suspensions $s_j$ used in the filtration have the same congruence class mod $k$. A connective spectrum is \emph{$k$-sparse} for $k$ an integer if it has a cell structure with only cells in a particular congruence class mod $k$.
\end{restatethis}
If $Z$ is a connective spectrum then $\EO\sm Z$ is cellular. A cellular $\EO$-module has an Atiyah-Hirzebruch spectral sequence. If a spectrum $Z$ is $k$-sparse for $k$ a divisor of $2p^2n^2$, then $\EO\sm Z$ is $k$-sparse. To show that a connective spectrum is $k$-sparse it suffices to check that $\HFp_*(Z)$ is concentrated in a single congruence class mod $k$.

Given an $\EO$-module $M$, we get an associated $K_*[C_p]$-module $\EEO_*(M)/\mfrak$, which has a decomposition into a sum of indecomposable $K_*[C_p]$-modules. We call the $\EO$-module ``algebraic'' if this splitting lifts to a splitting of $M$ into the standard $\EO$-modules $\EO\sm X_l$.
\begin{restatethis}{defn}{defn-algebraic-EO-theory}
An $\EO$-module $M$ is \emph{algebraic} if $M\simeq \EO\sm \bigwsum \susp^{s_i} X_{l_i}$. A spectrum $Z$ \emph{has algebraic $\EO$ theory} if $\EO\sm Z$ is algebraic.
\end{restatethis}
An algebraic $\EO$-module is evidently cellular. A cellular $\EO$-module $M$ is algebraic if and only if all differentials in the Atiyah-Hirzebruch spectral sequence for $M$ vanish except for the $d_{2n}$ differential.

When $p=3$, our definition of an algebraic $\EO$-module is closely related to Meier's definition of a ``standard vector bundle'' \cite[Definition 3.9]{relatively-free}. In Meier's nomenclature a standard vector bundle is an $E_*[G]$-module that is isomorphic to $\EEO_*(M)$ for some algebraic $\EO$-module $M$.

In \Cref{subsec:homotopy-type-of-algebraic-EO-mod}, we prove \Cref{splitting-from-E-theory} that if $M$ is an algebraic $\EO$-module then $\EEO_*(M)/\mfrak$ determines $M$ up to lost information about shifts.
We show in \Cref{splitting-from-HFp} that if $Z$ is a spectrum with algebraic $\EO$ theory, the $P^1$ action on $\HFp_*(Z)$ determines the homotopy type of $\EO\sm Z$.
We also show in \Cref{algebraic-EOmod-closed-under-union} that a ``union'' of algebraic $\EO$-modules is algebraic.
In \Cref{subsec:checking-an-EO-module-is-algebraic} we produce conditions to check that $\EO$-modules are algebraic.
We show in \Cref{2n-sparse-algebraic-EO-theory} that a $2n$-sparse spectrum has algebraic $\EO$ theory and we show in \Cref{2p-sparse-algebraic-EO-theory} that a $2p$-sparse $\EO$-module is algebraic.  In \Cref{subsec:orientations} we prove the results quoted in the introduction. None of the material after \Cref{subsec:orientations} is necessary to prove the main results quoted in the introduction.

In \Cref{subsec:frees-split} we show that if $M$ is an $\EO$-module such that $\EEO_*(M)$ is a projective $E_*$-module and $\EEO_*(M)/\mfrak$ has a free $\KCp$ submodule then there is a splitting $M\simeq \EO\sm X_p \vee N$.
In \Cref{subsec:smash-product-of-algebraic-EO-modules} we prove a formula for the smash product of algebraic $\EO$-modules and show that algebraic $\EO$-modules are closed under smash product.

\subsection{Determining the homotopy type of an algebraic \texorpdfstring{$\EO$}{EO}-module}
\label{subsec:homotopy-type-of-algebraic-EO-mod}
In this section we'll show that if $M$ is an algebraic $\EO$-module, the splitting of $M$ can be deduced from the $G$-module decomposition of $E^{\EO}_*(M)/\mfrak$, up to some lost information about shifts. We then show that if $Z$ is a spectrum with algebraic $\EO$ theory, the splitting of $\EO\sm Z$ can be deduced from the $P(1)_*$-comodule structure of $\HFp_*(Z)$.

If $M$ is an algebraic $\EO$-module then in particular it is torsion free so $E_2\!\AHSS(M)$ is a free $\EO_*$-module.

\begin{lem}
Let $g\in G$ be an element of order $n^2$. Let $v\in K_{*}(S^{2k+\epsilon})$ be a generator for $\epsilon\in\{0,1\}$. There is a primitive $n^2$ root of unity $\omega$ independent of $k$ and $\epsilon$ such that $g_*(v)=\omega^k v$.
\end{lem}

\begin{proof}
Let $\Gamma$ be the Lubin Tate formal group associated to $E_n$. Suppose that $h\in \G$ has power series representation $a_0 s +^{\G} \sum_{i\geq 1}^{\G} a_i s^{p^i}$.
Let $\BPP$ be periodic $\BP$-theory so that $\BPP_*\BPP$ parameterizes not-necessarily-strict $p$-typical power series.
Note that $\BPP_*\BPP = \Zp[v_0^{\pm},v_1,\ldots][t_0^{\pm},t_1,\ldots]$. Let $Z$ be a spectrum and let $z\in \BPP_*(Z)$. Write $\supK{z}$ for the image of $z$ in $K_*(Z)$. Suppose that $\Psi(z) = t_0^k\otimes z + \sum_{i} \theta_i \otimes z_i $ Then $h_*(\supK{z}) = a_0^k\supK{z} + \sum_{i>1} \theta_i(g)\supK{z}_i$.

In particular, the element $g$ of order $n^2$ has power series expansion $\omega s+^{\G} \sum_{i\geq 1}^{\G} a_i s^{p^i}$ where $\omega$ is some primitive $n^2$ root of unity. The coaction on a generator $v$ of $\BPP_*S^{2k+\epsilon}$ is given by $\Psi(v) = t_0^k\otimes v$. It follows that $g(v) = \omega^kv$.
\end{proof}

For $s\in \Z/2n^2$ write $\susp^{s}V_{l}$ for the $\KG$-module $K_*(\susp^s X_{l})$.
\begin{cor}
\label{splitting-from-E-theory}
Suppose that $M$ is an algebraic $\EO$-module and $\EEO_*(M)/\mfrak\cong \bigoplus_{k\in T} \susp^{\sbar_k}V_{l_k}$ as $\KG$-modules, where $T$ is some index set and $\sbar\in \Z/2n^2$. Then $M\simeq \EO\sm  \bigvee_{k\in T} \susp^{2ns_k} X_{l_k}$ where $s_k$ is some particular lift of $\sbar_k$ to $\Z/2p^2n^2$.
\end{cor}

So we can use $\EEO_*(M)$ to determine an algebraic $\EO$-module $M$ up to loss of information about shifts. We show now that the Atiyah-Hirzebruch spectral $E_{2n}$ page recovers the full homotopy type of an $\EO$-module.
\begin{lem}
\label{algebraic-EOmod-determined-by-E2n-AHSS}
Suppose that $M$ and $N$ are two algebraic $\EO$-modules, and suppose there is an isomorphism of bigraded $\EO_*$-modules $f\colon E_2\!\AHSS(M)\to E_2\!\AHSS(N)$. Let $E_2\!\AHSS(M)\cong \EO_*\{[x_i]\}_{i\in S}$ and suppose that $d_{2n}(f([x_i]))=f(d_{2n}([x_i]))$ for all $i\in S$. Then $M$ and $N$ are equivalent.
\end{lem}

\begin{proof}
An algebraic $\EO$-module $M$ is of the form $\EO\sm \bigwsum_{i\in S} \susp^{s_i}X_{l_i}$ where $s_i\in \Z/2p^2n^2$ and $l_i\in \{1,\ldots,p\}$.
The lengths and shifts are both determined by the $E_{2n}\!\AHSS(M)$ -- a summand of the form $\EO\sm \susp^{s_i}X_{l_i}$ corresponds to a summand of $E_{2n}\!\AHSS(M)$ which is an $l_i$-dimensional $\EO_*$-module on generators $\{[x_{0}],\ldots,[x_{i-1}]\}$ with differential $d_{2n}([x_{k}])=\alpha[x_{k-1}]$ for $k>0$ and $[x_{0}]$ a permanent cycle in the $s_i$ stem.
A decomposition of $M$ into summands of the form $\EO\sm \susp^{s_i}X_{l_i}$ corresponds exactly to a decomposition of $E_{2n}\!\AHSS(M)$ into summands of the form $E_{2n}\!\AHSS(\EO\sm \susp^{s_i}X_{l_i})$. It follows that $E_{2n}\!\AHSS(M)$ determines $M$.
\end{proof}

\begin{restatethis}{thm}{splitting-from-HFp}
Let $Z$ be a spectrum with algebraic $\EO$ theory. Decompose $\HFp_*(Z)$ into indecomposable $P(1)_*$-comodules, say $\HFp_*(Z)\cong \bigoplus_{i\in T} \susp^{s_i}W_{l_i}$ where $T$ is some index set. Then $\EO\sm Z\simeq \EO\sm \bigvee \susp^{s_i}X_{l_i}$.
\end{restatethis}

\begin{proof}
Pick an integral lift of the map $\HFp_*(Z)\to \bigoplus_{i\in T} \susp^{s_i}W_{l_i}$ to a map $\HZ_*(Z)\to \HZ_*\left(\bigvee \susp^{s_i}X_{l_i}\right)$.
This map induces an isomorphism
\[f\colon E_2\!\AHSS(\EO\sm Z)\to E_2\!\AHSS\left(\EO\sm \left(\bigvee \susp^{s_i}X_{l_i}\right)\right).\]
I claim that for $\{[x_i]\}_{i\in S}$ a basis for $\HZ_*(Z)$, we have $d_{2n}(f([x_i])) = f(d_{2n}(x_i))$.

Consider the map $\AHSS(Z)\to \AHSS(\EO\sm Z)$. Because $Z$ is torsion free, the shortest possible Atiyah-Hirzebruch differential is a $d_{2n}$ which is detected by the $P^1$ action on $\HFp_*(Z)$. Let $\rho$ be the reduction map $\HZ\to \HFp$.
Suppose that $x\in \HZ_{i}(Z)$ and $y\in \HZ_{i-2n}(Z)$. If $P^1(\rho(x)) = c\rho(y)$ where $c\in\Fp$ is some constant, then $d_{2n}([x])=c\alpha[y]$.
Since $1$ and $\alpha$ have nontrivial image in $\EO_*$ we deduce a differential $d_{2n}([x])=c\alpha[y]$ in $E_{2n}\!\AHSS(\EO\sm Z)$. Because the map $f\colon \HZ_*(Z)\to \HZ_*\left(\bigvee \susp^{s_i}X_{l_i}\right)$ was a lift of a map that commutes with $P^1$, we have also that $P^1(\rho(f(x))) = \rho(f(y))$.
We deduce that $f(d_{2n}([x]))=c\alpha f([y])=d_{2n}(f([x]))$. The hypotheses of \Cref{algebraic-EOmod-determined-by-E2n-AHSS} are met and we conclude that $\EO\sm Z\simeq \EO\sm \bigvee \susp^{s_i}X_{l_i}$.
\end{proof}

\begin{cor}
\label{B-iso-implies-EO-equiv}
If $X$ and $Y$ are connective spectra with algebraic $\EO$ theory and $\HFp_*(X)\cong \HFp_*(Y)$ as $P(1)_*$-comodules, then $\EO\sm X\simeq \EO\sm Y$.
\end{cor}

Now we show that algebraic $\EO$-modules are closed under ``unions.''
\begin{prop}
\label{algebraic-EOmod-closed-under-union}
Suppose that $M_1\to M_2\to \cdots$ is a diagram of algebraic $\EO$-modules such that each map $M_i\to M_{i+1}$ induces an injection $\EEO_*(M_i)/\mfrak\to \EEO_*(M_{i+1})/\mfrak$. Then $\hocolim M_i$ is an algebraic $\EO$-module.
\end{prop}

\begin{proof}
Write $\EEO_*(M)/\mfrak\cong \bigoplus_{j\in S} V_{l_j}$ as $\KCp$-modules. To show that $M$ is algebraic, we need to show that this splitting lifts to a splitting of $M$. Pick some summand $V_{l_j}$ of $\EEO_*(M)/\mfrak$. Because $V_{l_j}$ is finite dimensional, for some $i$ sufficiently large, $\EEO_*(M_i)/\mfrak\to \EEO_*(M)/\mfrak \to V_{l_j}$ is a surjection. Since the map $\EEO_*(M_i)/\mfrak\to \EEO_*(M)/\mfrak$ is an injection, we deduce that there is a splitting $\EEO_*(M_i)\cong V_{l_j}\bigoplus W$ and because $M_i$ is algebraic, this lifts to a splitting $M_i\simeq \EO\sm \susp^{s_j}X_{l_j}$ for some $s_j\in \Z/2p^2n^2$. This gives a map $\iota_j\colon \EO\sm \susp^{s_j}X_{l_j}\to M$ which induces the inclusion $V_{l_j}\to \EEO_*(M)/\mfrak$. Summing the maps $\iota_j$ as $j\in S$ varies gives a map from $\EO\sm\bigwsum_{j\in S} \susp^{s_j}X_{l_j}\to M$ which induces an isomorphism $E_*\left(\bigwsum_{j\in S} \susp^{s_j}X_{l_j}\right)/\mfrak\to \EEO_*(M)/\mfrak$. By Nakayama's lemma, this also induces an isomorphism $E_*\left(\bigwsum_{j\in S} \susp^{s_j}X_{l_j}\right) \to \EEO_*(M)$ which implies that there is an isomorphism of $E_2$ pages $\HFPSS\left(\EO\sm \bigwsum_{j\in S} \susp^{s_j}X_{l_j}\right) \to \HFPSS(M)$. It follows that the map $\EO\sm \bigwsum_{j\in S} \susp^{s_j}X_{l_j} \to M$ is an equivalence, and hence $M$ is algebraic.
\end{proof}

\subsection{A brief review of the homotopy fixed point spectral sequence for \texorpdfstring{$\EO$}{EO}}
\label{subsec:review-of-EO*}
There is a map $\ANSS(S^0)\to\HFPSS(\EO)$ from the Adams Novikov spectral sequence for the sphere to the homotopy fixed point spectral sequence $H^*_G(E_*)\Rightarrow \EO_*$.
Hopkins and Miller computed the homotopy fixed point spectral sequence for $\EO$ up to some permanent cycles on the zero line. The $E_2$ page is isomorphic to $\Fp[\alpha,\beta,u^{\pm}]$ in positive filtration, where $\alpha\in(2n-1,1)$ and $\beta\in(2pn-2,2)$ are the images of $\alpha_1$ and $\beta_1$ in $\ANSS(S^0)$ and $u\in(2pn^2,0)$ is a norm class. There are two differentials $d_{2n+1}(u)=\alpha\beta^n$ and $d_{2n^2+1}(\alpha u^n)=\beta^{n^2+1}$ in $\HFPSS(\EO)$. All other differentials are generated by these two using the Leibniz rule. The $E_{\infty}$ page has a horizontal vanishing line at filtration $2n^2+2$. The element $u^p$ is a permanent cycle, which gives $\EO$ theory a $2p^2n^2$ periodicity. The homotopy of $\EO_*$ for $p=3$ and $p=5$ is illustrated in \Cref{fig:homotopy-of-EO}. See the account of this spectral sequence in \cite[section 2]{nave} for more details.

\subsection{Conditions for an \texorpdfstring{$\EO$}{EO}-module to be algebraic}

For $Z$ a connective spectrum, we use the cellular filtration of $Z$ to get filtrations of $\BP_*(Z)$ and $\E_*(Z)$ which gives algebraic Atiyah-Hirzebruch spectral sequences. If $Z$ is torsion free, this has the following form:
\begin{align*}
\algAHSS(Z)\colon \Ext_{\BP_*\BP}(\BP_*,\BP_*)\otimes \HZ_*(Z)&\Rightarrow \Ext_{\BP_*\BP}(\BP_*,\BP_*(Z))\\
\algAHSS(\EO \sm Z)\colon \phantom{\,\Ext_{\BP_*\BP}(\BP_*,\BP_*)} \mathllap{H^*_G(E_*)}\otimes \HZ_*(Z)&\Rightarrow H^*_G(\E_*(Z))
\end{align*}
The map $\ANSS(Z)\to \HFPSS(\EO\sm Z)$ induces a map $\algAHSS(Z)\to \algAHSS(\EO\sm Z)$.
The homology of $X_l$ is $\HZ_*(X_l) = \Z\{x_0,\ldots,x_{l-1}\}$ with $x_i$ in degree $2in$, so the $E_2$ page of $\algAHSS(\EO\sm X_l)$ is isomorphic modulo trace classes to $\Fp[\alpha,\beta,u^{\pm}]\{[x_0],\ldots, [x_{l-1}]\}$.

\label{subsec:checking-an-EO-module-is-algebraic}
For a connective spectrum $Z$, denote by $\HI_d(Z)$ the Hurewicz image of $\pi_d(Z)\to \EO_d(Z)$.
\begin{prop}
\label{Hurewicz-image-Xl}
Let $1\leq l < p$ and let $\alpha^{(l)} \in \EO_{2nl - 1}(X_l)$ be the Hurewicz image of the attaching map for the top cell of $X_{l+1}$. Then $\alpha^{(l)}$ is nonzero and spans the Hurewicz image in $\EO_{2nl-1}(X_l)$. If $k\neq l$ then the Hurewicz image in $\EO_{2kn-1}(X_l)$ is zero. If $l<p$, projection onto the top cell $X_{l}\to S^{2n(l-1)}$ induces an isomorphism $\HI_{2nk-1}(X_l)\to \HI_{2nk-1}(S^{2n(l-1)})$.
\end{prop}

\begin{proof}
Consider the map of spectral sequences $\ANSS(X_l) \to \HFPSS(\EO\sm X_l)$. In the degree we are considering, $\HFPSS(\EO\sm X_l)$ only contains elements in filtration one and in the degree we're considering $\ANSS(X_l)$ contains no elements in filtration zero, so no filtration jumping can happen and it suffices to understand the image of the map $\ANSS(X_l) \to \HFPSS(\EO\sm X_l)$. We will first handle the case when $l=1$ and $X_l=S^0$, and then we will use algebraic Atiyah-Hirzebruch spectral sequences to deduce the case for larger $l$.

Suppose that $l=1$. We have that $\pi_{2n-1}(\EO) = \Fp\{\alpha_1\}$. We want to show that if $k\neq 1$ then $\HI_{2nk-1}(\EO) = 0$.
We will show that this is true on the $E_{2n+1}$ page of $\ANSS(S^0) \to \HFPSS(\EO)$.
Refer to \vref{fig:homotopy-of-EO} for the $E_{\infty}$ page of $\HFPSS(\EO)$ for $p=3$ and $p=5$.
Because the only element of the $0$-line of $\ANSS(S^0)$ is $1\in \pi_0$, nothing else in the zero line of $\HFPSS(\EO)$ is in the Hurewicz image, so we need only study positive filtration.
The $E_2$ page of $\HFPSS(\EO)$ is isomorphic in positive filtration to $\Fp[\alpha,\beta,v]$ where $\alpha\in(2n-1,1)$, $\beta\in(2pn-2,2)$ and $v\in (2pn^2,0)$. So $|\alpha|\equiv -1$,  $|\beta|\equiv -2$ and $|v|\equiv 0\pmod{2n}$.
The elements of the $E_2$ page in degree $-1\pmod{2n}$ are $\alpha\beta^{in}v^j$. There is a differential  $d_{2n}(v)=\alpha\beta^{n}$,  so all elements on the $E_\infty$ page in degree $-1\pmod{2n}$ are of the form $\alpha v^j$ and hence are in the 1-line.
I claim that if $j\neq 0$ then $\alpha v^j$ is not in the image of the map on $E_2$ pages.  The only element of the Novikov 1-line in the degree of $\alpha v^j$ is $\alpha_{npk+1}$.
There is a Massey product in the Novikov $E_2$ page $\alpha_{npk+1} = \toda{\alpha_{npk},p,\alpha_1}$ with indeterminacy in filtration greater than 1.
Because the homotopy fixed point spectral sequence contains no elements in the same stem as $\alpha v^j$ in filtration greater than 1, the indeterminacy of this Massey product maps to zero in the $E_2$ page of the homotopy fixed point spectral sequence. By sparsity $\alpha_{npk}\mapsto 0$. It follows that $\alpha_{npk+1}\mapsto 0$ too, and $\alpha v^j$ is not in the Hurewicz image. This settles the $l=1$ case.

Now we use the algebraic Atiyah Hirzebruch spectral sequence to reduce the case $l>1$ to the case $l=1$. \Vref{fig:algAHSS-EO2-X3} is an illustration of $\algAHSS(\EO\sm X_3)$.
Because the cells of $X_l$ are in degrees congruent to $0 \pmod{2n}$, and because $\alpha\beta^n = 0\in \EO_*$, the only elements of $\algAHSS(\EO\sm X_l)$ in degrees congruent to $-1\pmod{2n}$ are those of the form $\alpha v^j [x_i]$. By the $l=1$ case, the only such elements that are hit in the $E_2$ page of the map $\algAHSS(X_l)\to \algAHSS(\EO\sm X_l)$ are $\alpha[x_i]$. Since each attaching map in $X_l$ is given by an $\alpha$, there are Atiyah-Hirzebruch differentials $d_{2n}([x_{i+1}])=\alpha[x_i]$ so all of these elements are zero in homotopy except for $\alpha[x_l]$. If $l<p$, this is a permanent cycle which detects $\alpha^{(l)}$.
\end{proof}

\sseqset{Zclass/.style={rectangle}, pZclass/.style={rectangle,fill=none}}

\NewSseqCommand\alphaclass {} {\class(\lastx+3,\lasty+1)}

\NewSseqGroup\EO {} {
    \foreach \v in {0,...,3}{
        \class(24*\v,0)
        \alphaclass
        \structline
        \foreach \b in {1,...,9}{
            \class(\lastx{1}+10,\lasty{1}+2)
            \structline(\lastclass2)
            \class(\lastx{1}+10,\lasty{1}+2)
            \structline(\lastclass2)
	    \structline
        }
       \classoptions[Zclass](24*\v,0)
    }
}

\begin{sseqdata}[name=algAHSS-EO2-X3,Adams grading,
    classes=fill,
    width=400pt,
    height=200pt,
    run off struct lines=-,
    y range={0}{7},x range={0}{72},
    x tick step=5,]
\EO(0,0)
\EO[blue](4,0)
\EO[red](8,0)
\foreach \v in {0,1,2}
\foreach \c in {1,2}
\foreach \b in {0,...,8}{
	\d1(4*\c+24*\v+10*\b,2*\b)
}

\foreach \v in {0,1,2}
\foreach \b in {0,...,8}{
	\d1(8+3+24*\v+10*\b,2*\b+1)
}
\end{sseqdata} 
\afterpage{
    \clearpage
    \begin{figure}[H]
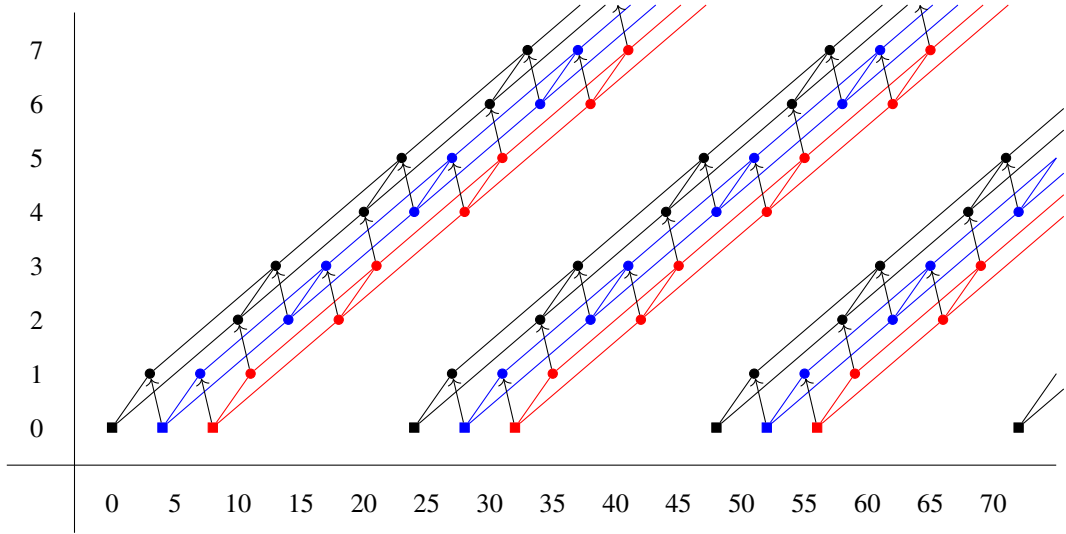

    \printpage[name=algAHSS-EO2-X3]
    \caption{The algebraic AHSS for $\EO_*X_3$ at $p=3$}
    \label{fig:algAHSS-EO2-X3}
    \end{figure}
    \leavevmode
}
\begin{lem}
\label{HI-subspace}
Suppose that $Y$ is a connective $2n$-sparse spectrum with cells in degrees $0\pmod{2n}$. Then the map $Y\to Y_{(2nk)}$ induces an injection $\HI_{2n(k+1)-1}(Y)\to \HI_{2n(k+1)-1}(Y_{(2nk)})$. The map $Y_{(2nk)}^{(2nk)}\to Y_{(2nk)}$ induces a surjection $\HI_{2n(k+1)-1}\left(Y_{(2nk)}^{(2nk)}\right)\to \HI_{2n(k+1)-1}(Y_{(2nk)})$. Thus $\HI_{2n(k+1)-1}(Y)$ is a subquotient of $\HI_{2n(k+1)-1}\left(Y_{(2nk)}^{(2nk)}\right)$.
\end{lem}

\begin{proof}
Consider $\AHSS(\EO\sm Y)$. The Hurewicz image $\HI_{2(k+1)n-1}(S^0)$ is zero unless $k=0$ and $\HI_{2n-1}(S^0) = \Fp\{\alpha\}$ so the classes in $\HI_{2in-1}Y$ are exactly those detected by an element of the form $\alpha[x]$ for $x\in \HZ_*Y$. Since the degree of $\alpha$ is $2n-1$, the degree of $\alpha[x]$ is $2n-1 + |x|$. For $\alpha[x]$ to be in degree $2n(k+1)-1$, the homology class $x$ should be in degree $2nk$. We deduce that every class in $\HI_{2n(k+1)-1}(Y)$ is detected in Atiyah-Hirzebruch filtration $2nk$ so that $\HI_{2n(k+1)-1}(Y)$ is a subquotient of $\HI_{2n(k+1)-1}\left(Y_{(2nk)}^{(2nk)}\right)$.
\end{proof}

\begin{lem}
\label{cone-alphas-algebraic}
Let $M=\EO\sm\bigvee_{i\in T} \susp^{s_i} X_{l_i}$ be an algebraic $\EO$-module. Suppose that $f\colon \bigwsum_{S}\susp^{s_i}\EO\to M$ is some $\EO$ module map such that the homotopy class of each component $\susp^{s_i}\EO\to M$ is contained in $\Fp\{\susp^{s_i}\alpha^{(l_i)}\}_{i\in U}\subseteq \pi_s(M)$ where $U\subseteq T$ is the subset of $i\in T$ such that $l_i<p$ and $s_i+2n(l_i-1)=s$. Then $C(f)$ is an algebraic $\EO$-module.
\end{lem}

\begin{proof}
First suppose we are only attaching one cell along a map $\susp^s\EO\to M$.
By assumption, $f_*\in \EO_{s}M$ is some linear combination $\sum_{i\in U}a_{i}\susp^{s_i}\alpha^{(l_i)}$. If all $a_i$ are zero, then $C(f)\simeq M\vee \susp^{s+1}\EO$ is algebraic. Otherwise, suppose that $a_1\neq 0$  and that for all $i\in U$ such that $a_i\neq 0$,  $l_1\geq l_i$. By \Cref{filtered-automorphisms-sums-of-X_l} there is an automorphism $\phi$ of $M$ such that $\phi_*^{-1}(\alpha^{(l_1)})=f_*$ and $\phi^{-1}_*(\alpha^{(l_i)})=\alpha^{(l_i)}$ for $i>1$.  Then $\phi\circ f =\susp^{2ns_1} \alpha^{(l_1)}$ so that
\[C(f)\simeq C(\phi\circ f ) = \EO\sm \susp^{s_1} X_{l_1+1} \vee \bigvee_{i\in T\setminus\{1\}}\susp^{s_i} X_{l_i}.\]
We conclude that $C(f)$ is an algebraic $\EO$-module.

Now consider the case where we are attaching multiple cells, say $f\colon \bigwsum_{S} \susp^{s_i}\EO\to M$. Filter $C(f)$ by picking a total order on $S$ and letting $f_i$ be the restriction of $f$ to $\bigwsum_{j\leq i} \susp^{s_j}\EO$. Then let $N_i = C(f_i)$. Note that $N_i$ is the cofiber of the composite $\susp^{s_i}\EO\to \bigwsum_{S} \susp^{s_i}\EO\to M \to N_{i-1}$, and this composite satisfies the hypotheses of the lemma so $N_i$ is an algebraic $\EO$-module for each $i$. Since $C(f)=\hocolim_i N_i$, by \Cref{algebraic-EOmod-closed-under-union} $C(f)$ is algebraic.
\end{proof}

To finish the proof of \Cref{cone-alphas-algebraic} we need the following lemma:
\begin{lem}
\label{filtered-automorphisms-sums-of-X_l}
Fix $d$ an integer and let
\[X= \bigvee_{i\in S} \susp^{d-2nl_i} X_{l_i}\]
where $l_i<p$ for all $i\in S$.
Let $U=\EO_{d-1}(X)=\Fp\{\alpha^{(l_i)}\}_{i\in S'}$.
Filter $U$ by $U_{l} = \Fp\Set{\alpha^{l_i} \given l_i<l}$ and suppose that $f\colon \EO_{d-1}(X)\to \EO_{d-1}(X)$ is a filtered endomorphism.
Then there is an endomorphism $\ftwee\colon X\to X$ that induces $f$ on $\EO_{d-1}(X)$.
\end{lem}
\begin{proof}
The collapse map $p\colon X_l \to \susp^{2n} X_{l-1}$ sends $\alpha^{(l)}\mapsto \alpha^{(l-1)}$. This follows from the commutativity of the following diagram, where the columns are cofiber sequences:
\[\begin{tikzcd}[column sep=-1.5em]
&S^{2n(l+1)-1}\dlar["\alpha^{(l)}"']\drar["\alpha^{(l-1)}"] &\\
X_l \ar[rr,"p"]\dar &&\susp^{2n} X_{l-1}\dar\\
X_{l+1} \ar[rr,"p"] && \susp^{2n} X_l
\end{tikzcd}\]
The multiplication by $c$ map $c\colon X_l\to X_l$ sends $\alpha^{(l)}\mapsto c\alpha^{(l)}$. If $X =\bigvee_{l=1}^{p-1}\susp^{d - 2nl}\bigvee_{S_l} X_{l}$, any block upper triangular matrix of integers represents an endomorphism of $X$, and this endomorphism of $X$ induces the corresponding filtered endomorphism of $\EO_{d-1}(X)$.
\end{proof}

\begin{restatethis}{thm}{2n-sparse-algebraic-EO-theory}
Let $Z$ be a connective $(2p-2)$-sparse spectrum. Then $Z$ has algebraic $\EO$ theory.
\end{restatethis}

\def\fbar{\overline f}
\def\ftilde{\widetilde{f}}
\begin{proof}
By \Cref{algebraic-EOmod-closed-under-union} we may argue by cellular induction on $Z$. If $Z$ has one cell, the statement is immediate. Suppose that $Y$ is a connective $2n$-sparse spectrum with cells in dimension less than or equal to $2nk$.  Suppose also that $Y$ has algebraic $\EO$ theory, say $\EO\sm Y\simeq  \EO\sm\bigvee_{i\in T} \susp^{s_i} X_{l_i}$ and that $Z$ is the cofiber of $f\colon \bigwsum_{S} S^{2nk-1}\to Y$. It follows that $\EO\sm Z$ is the cofiber of $\EO\sm f$.

By \Cref{HI-subspace}, $\HI_{2nk-1}(Y)$ is a subspace of $\HI_{2nk-1}(Y_{2n(k-1)})\cong \bigoplus_{T'}\Fp[\susp^{2nk}\alpha]$ where $T'\subseteq T$ is the subset of $i\in T$ such that $s_i+2nl_i=2nk$. Under the isomorphism $\pi_*(\EO\sm Y) \cong \pi_*\left(\EO\sm \bigvee_{i=1}^d \susp^{s_i} X_{l_i}\right)$, the module $\HI_{2nk-1}(Y)$ is the subspace of classes of $\Fp\{\susp^{s_i}\alpha^{(l_i)}\}_{i\in S}$ where $s_i+2nl_i=2nk$, and such that the element $\alpha_1[x]\in \AHSS(Y)$ that maps to the element of $\AHSS(\EO\sm Y)$ that detects $\alpha^{(l_i)}$ is a permanent cycle. By \Cref{cone-alphas-algebraic}, it follows that $C(f)$ has algebraic $\EO$ theory.
\end{proof}

\begin{restatethis}{thm}{2p-sparse-algebraic-EO-theory}
Suppose that $M$ is a $2p$-sparse cellular $\EO$-module. Then $M$ is algebraic. In fact, $M\simeq \bigwsum \susp^{s_i} \EO$. If $Z$ is a $2p$-sparse connective spectrum, then $Z$ has algebraic $\EO$ theory.
\end{restatethis}

\begin{proof}
It suffices to show that $\EO_{2pk-1}=0$ for all $k$. The $0$-line of $\HFPSS(\EO)$ consists of the fixed points of the action of $G$ on $E_*$. Since $E_*$ is even, so is the zero line. We need to show that there are no elements in positive filtration in the $2pk-1$ stem on the $E_{\infty}$ page of $\HFPSS(\EO)$. Recall that the $E_2$ page of $\HFPSS(\EO)$ is isomorphic to $\Fp[\alpha,\beta,v^{pm}]$ in positive filtration and that the elements of the $E_2$ page representing nonzero odd degree homotopy elements are all of the form $\alpha\beta^i v^{j}$ where $i<p-1$. These are in degrees $\deg{\alpha}=2p-3\equiv -3\pmod{2p}$, $\deg{\beta}=2pn-2\equiv -2\pmod{2p}$ and $\deg{u}=2pn^2\equiv 0\pmod{2p}$. Now $\deg{\alpha\beta^i}\equiv -3-2i\pmod{2p}$ so if $\deg{\alpha\beta^i}\equiv -1\pmod{2p}$ then $-3-2i\equiv -1\pmod{2p}$ implies that $2i\equiv -2 \pmod{2p}$ so $i=(p-1)+pk$ but then $\alpha\beta^i=0\in \EO_*$.
\end{proof}

\begin{prop}
\label{EO-splitting-small-range}
Let $X$ be a torsion free connective spectrum with cells in degrees between $k$ and $k+2pn-2$. Suppose that $M$ is a retract of $\EO\sm X$. Then $M$ has algebraic $\EO$ theory.
\end{prop}
In order to prove this, we use the following lemma, which has messier hypotheses:
\begin{lem}
\label{EO-splitting-small-range-2}
Suppose that $M$ is a cellular $\EO$-module, say $M=\hocolim M_i$ and $M_{i+1}$ is the cone of some map $f_i\colon\bigwsum_{j\in S_i}\susp^{s_{ij}}\EO \to M_i$ where $s_{ij}\in \Z/2p^2n^2$. Suppose that $E\sm_{\EO}f\simeq 0$. Suppose also that there are integral lifts $\widetilde{s}_{ij}$ such that for some $k\in \Z$ and for all $i$, $j\in S_i$ we have $k\leq \widetilde{s}_{ij}\leq k+2pn-2$. Then $M$ is an algebraic $\EO$-module.
\end{lem}

\begin{proof}
We argue by cellular induction. Suppose that $M_i$ is an algebraic $\EO$-module, say $M_i\simeq \EO\sm\bigvee_{t\in T} \susp^{s_t} X_{l_t}$. We need to show that the cone of $f_i\colon\bigwsum_{j\in S_i}\susp^{s_{ij}}\EO \to M_i$ is algebraic. It suffices by \Cref{cone-alphas-algebraic} to show that the homotopy of each component map $f_{ij}\colon\susp^{s_{ij}}\EO\to M_i$ lies in the subgroup of $\pi_{s_{ij}}(M)$ generated by $\Fp\{\susp^{s_t}\alpha^{(l_t)}\}_{t\in T'}$ where $T'\subseteq T$ is the subset of $t\in T$ such that $s_t+2n(l_t-1)=s_{ij}$ and $l_t<p$. By assumption, $f_{ij}$ is detected in $\AHSS(M_i)$ by some element $\theta[x]$ where $\theta\in \EO_*$ and $x$ is some cell of $M_i$ such that the image of $\theta$ under $\EO\to E$ is zero. We also know that $|x| + |\theta| = s_{ij}$ and that $k\leq |x|, s_{ij}\leq 2pn-2$. It follows that $0\leq |\theta|\leq 2pn-2$. The only element of the kernel of $\EO_*\to E_*$ in these degrees is $\alpha$. We deduce that $f_{ij}$ is detected in the necessary subspace. By \Cref{cone-alphas-algebraic}, $M_{i+1}$ is algebraic.
\end{proof}

\begin{proof}[Proof of \Cref{EO-splitting-small-range}]
First note that $\EO\sm X$ satisfies the hypotheses of \Cref{EO-splitting-small-range-2} -- $\EO\sm X$ has a cellular filtration where each cell is in dimension between $k$ and $k+2pn-2$. Because $X$ is torsion free, $E\sm X^{(i)}$ splits as a sum of $E$-theories, so each attaching map $f_i\colon \bigwsum_{S_i}\susp^{i-1}\EO \to \EO\sm X^{(i-1)} \to \EO\sm X^{(i)}$ must be zero on $E$-theory.

Now let $M$ be a retract of $\EO\sm X$ and let $M_i$ be the corresponding retract of $\EO\sm X^{(i)}$. Since $E\sm X^{(i)}$ splits as a sum of $E$-theories, and $E\sm_{\EO}M_i$ is a retract of $E\sm X^{(i)}$, it follows that $E\sm_{\EO}M_i$ does too. Thus, the attaching maps to form $M_{i}$ from $M_{i-1}$ must be in the kernel of $\pi_*(M_{i-1})\to \EEO_*(M_{i-1})$. The dimensions of the cells of $M$ are still in the range from $k$ to $k+2pn-2$, so $M$ satisfies the hypotheses of \Cref{EO-splitting-small-range-2} as well and $M$ is an algebraic $\EO$-module.
\end{proof}

\subsection{Orientations}
\label{subsec:orientations}
Here we give proofs of the theorems quoted from this section in the introduction. At this point these arguments are straightforward.

\begin{lem}
\label{alg-EO-from-homology-impl-Chern-determined}
Let $Z$ be a space and suppose that every spectrum $Y$ with $\HFp_*(Y)$ isomorphic to $\HFp_*(Z)$ as $P(1)_*$-comodules has algebraic $\EO$-theory. Then $\EO$-orientability of complex bundles over $Z$ is Chern determined.
\end{lem}

\begin{proof}
Suppose that $V$ is a complex bundle over $Z$ with $\psi^{p-1}(V)=0\pmod{p}$. We need to show that $V$ is $\EO$-orientable. Since $\psi^{p-1}(V)=0\pmod{p}$, $\HFp_*(\Th(V))\cong \HFp_*(Z)$ as $P(1)_*$-comodules. By hypothesis, this implies that both $\Th(V)$ and $Z$ have algebraic $\EO$-theory and by \Cref{B-iso-implies-EO-equiv} $EO\sm Z$ and $\EO\sm \Th(V)$ are homotopy equivalent so $V$ is orientable.
\end{proof}

\begin{restatethis}{thm}{EO-orientation-2n-sparse}
Let $Z$ be a $(2p-2)$-sparse space. Then $\EO$-orientability of complex bundles over $Z$ is Chern-determined. In particular, let $\psi_{p-1}$ be the $(p-1)$st power sum polynomial over $Z$. Then a complex vector bundle $\xi$ over $Z$ is $\EO$-orientable if and only if $\psi_{p-1}\equiv 0\pmod p$.
\end{restatethis}

\begin{proof}
\Cref{2n-sparse-algebraic-EO-theory} says that the hypothesis of \Cref{alg-EO-from-homology-impl-Chern-determined} is satisfied.
\end{proof}

\begin{restatethis}{prop}{sparse-connective-implies-orientable}
Let $Z$ be a $(2p-2)$-sparse $2p$-connective space. Then every map $Z\to\BGLS$ is $\EO$-orientable.
\end{restatethis}

\begin{proof}
$Y_{4p-4}$ is $(2p-2)$-sparse so both $\Th(f)$ and $\susp^{\infty}_+Y_{4p-4}$ are $(2p-2)$-sparse. By \Cref{2n-sparse-algebraic-EO-theory} we deduce that they have algebraic $\EO$-theory. If $u$ is the Thom class in $\HFp^*(\Th(f))$ then for connectivity reasons, $P^1(u)=0$ so the Thom isomorphism respects the $P(1)^*$-module structure. The theorem follows by \Cref{B-iso-implies-EO-equiv}.
\end{proof}
We immediately deduce:
\restate{4p-4_orientation}

\subsection{Free \texorpdfstring{$E_*[C_p]$}{E\_*[Cp]} summands of \texorpdfstring{$\EEO_*(M)$}{E\string^{EO}\_*(M)} lift to summands of \texorpdfstring{$M$}{M}}
\label{subsec:frees-split}
We can split off copies of $\EO\sm X_p$ from an $\EO$-module $M$ without assuming that $M$ is sparse or induced:
\begin{restatethis}{prop}{EO-splitting-free}
Suppose that $M$ is an $\EO$-module such that $\EEO_*(M)$ is a projective $E_*$-module. The map
\[\pi_0 \EOmod(\EO\sm X_p, M) \to \Hom_{\EEO_*E}(\E_*(X_p),\EEO_*(M))\]
given by $f\mapsto \pi_0(E\sm_{\EO}f)$ is an isomorphism, natural in $M$. If $M$ is a finite module, the map
\[\pi_0 \EOmod(M,\EO\sm X_p) \to \Hom_{\EEO_*E}(\EEO_*(M),\E_*(X_p))\]
given by $f\mapsto \pi_0(E\sm_{\EO}f)$ is an isomorphism, natural in $M$. If $M$ is a cellular $\EO$-module and $\EEO_*(M) \cong E_*(\susp^s X_p)\oplus V'$ where $V'$ is some $E_*[G]$-module, then $M\simeq \EO\sm \susp^s X_p \vee M'$ for some $\EO$-module $M'$ with $\EEO_*(M')=V'$ as $E_*[G]$-modules.
\end{restatethis}

\begin{proof}
There is a relative Adams spectral sequence
\[\Ext_{\EEO_*E}(E_*,\EEO_* (DX_p\sm M))\Rightarrow \pi_*(DX_p\sm M) = \pi_*\EOmod(\EO\sm X_p,M).\]
Because $E_*(X_p)$ is $E_*$-free, there is a Kunneth isomorphism
\[\EEO_*(DX_p\sm M)\cong E_*(DX_p) \otimes_{E_*} \EEO_*(M).\]
The $\EEO_*E$ coaction on $E_*(DX_p)$ is free, so we see that $\EEO_*(DX_p\sm M)$ has a free coaction too. It follows that the spectral sequence collapses on the $E_2$ page and \[\pi_*\EO(\EO\sm X_p,M)\to \Hom_{\EEO_*E}(E_*,E_*(DX_p) \otimes_{E_*} \EEO_*(M))\]
is an equivalence. Because $E_*(X_p)$ is $E_*$-free, $E_*(DX_p)\cong E_*(X_p)^{\vee}$ and because $E_*(X_p)$ is a finite dimensional free $E_*$-module,
\[\Hom_{\EEO_*E}(E_*,E_*(DX_p) \otimes_{E_*} \EEO_*(M))\cong \Hom_{\EEO_*E}(E_*(X_p),\EEO_*(M)).\]

We know that $E \sm X_p = E \sm_{\EO} (\EO \sm X_p) = E \sm_{\EO} \EhC$ so $E_*(X_p) = \EEO_*(\EhC)$ is a summand of $\EEO_*E$ and hence is a relatively injective $\EEO_*E$-comodule.
Because $\EEO_*E$ is a cofree $\EEO_*E$-comodule, for $V$ some other $\EEO_*E$ comodule there is an isomorphism between $\EEO_*E \otimes_{E*} V$ with the standard ``diagonal'' coaction and $\EEO_*E\otimes_{E_*} V$ with coaction just on the left tensor factor, so we see that relative injectives form a tensor ideal.
Thus $E_*(DX_p) \otimes_{E_*} \EEO_*(M)$ is a relative injective and $\Ext_{\EEO_*E}(E_*,\EEO_* (DX_p\sm M))$ vanishes in positive degrees. It follows that the edge map $\pi_0 \EOmod(\EO\sm X_p, M) \to \Hom_{\EEO_*E}(\E_*(X_p),\EEO_*(M))$ is an isomorphism.

Since $\EEO_*M$ is projective as an $E_*$-module, there is also a relative Adams spectral sequence
\[\Ext_{\EEO_*E}(\EEO_*(M),E_*(X_p))\Rightarrow \pi_*\EOmod(\EEO_* (M),E_*(X_p)).\]
Because $E_*(X_p)$ is relatively injective, this is concentrated on the zero line. It follows that the Adams spectral sequence collapses and the edge map
\[\pi_0\EOmod(M,\EO\sm X_p)\to \Hom_{\EEO_*E}(\EEO_*(M),E_*(X_p))\]
is an isomorphism.

If $M$ is cellular and $\EEO_*(M) \cong E_*\left(\susp^s X_p\right)\oplus V'$ then there is a map $f:\EO\sm \susp^s X_p\to M$ inducing the inclusion of the summand $E_*\left(\susp^s X_p\right)$ on relative $E$-theory. Since $\EO\sm X_p$ is compact, $f$ factors as a map $f^{(k)}\colon \EO\sm X_p \to M^{(k)}$ where $M^{(k)}$ is any sufficiently large skeleton of $M$. Skeleta of $M$ are finite so each $f^{(k)}$ splits and these splittings can be chosen compatibly. We conclude that $f$ splits.
\end{proof}

\begin{prop}
Suppose that $Z$ is a spectrum with $E_*(Z)$ a free $E_*$-module and there is a splitting $\HFp_*(Z) = \susp^{s} W_p \oplus U$ as $P(1)_*$-comodules. Then there is a splitting $\EO\sm Z\simeq \susp^s \EO\sm X_p \vee M$ where $M$ is some $\EO$-module.
\label{frees-split}
\end{prop}

\begin{proof}
Consider the subquotient $Z^{(s+2n^2)}_{(s)}$ of $Z$. Because $2n^2\leq 2pn-n$, \Cref{EO-splitting-small-range} says that $Z^{(s+2n^2)}_{(s)}$ has algebraic $\EO$ theory. Since $\HFp_*\left(Z^{(s+2n^2)}_{(s)}\right) = \susp^s W_p \oplus U'$ there is a splitting $\EO\sm Z^{(s+2n^2)}_{(s)}\simeq \susp^k \EO\sm X_p \vee \EO\sm Z'$.
This implies that 
\[E\sm Z^{(s+2n^2)}_{(s)} \simeq E\sm_{\EO} \EO\sm Z^{(s+2n^2)}_{(s)} \simeq \susp^s E\sm X_p \oplus E\sm Z',\] 
so $E_*\left(Z^{(s+2n^2)}_{(s)}\right) \cong E_*(\susp^s X_p) \oplus E_*(X)$ as $E_*[G]$-modules.
Thus $E_*(\susp^sX_p)$ is a subquotient of $E_*(Z)$ and $E_*(\susp^sX_p)$ is a free $E_*[C_p]$-module so this subquotient is split: $E_*(Z) \cong E_*(\susp^s X_p) \oplus E_*(Z')$. By \Cref{EO-splitting-free}, there is a splitting $\EO\sm Z \simeq \EO\sm \susp^{s'}X_p\wsum M'$ where $s'\equiv s \pmod{2n^2}$ and $\EEO_*(M')\cong E_*(Z')$. By comparing Atiyah-Hirzebruch spectral sequences as in the proof of \Cref{splitting-from-HFp} we see that $s'\equiv s \pmod{2p^2n^2}$.
\end{proof}

\subsection{A formula for the smash product of algebraic \texorpdfstring{$\EO$}{EO}-modules}
\label{subsec:smash-product-of-algebraic-EO-modules}
Recall that $P(1)_* = \Fp[t]/(t^p)$ with $\Delta(t) = t\otimes 1 + 1\otimes t$, so that $P(1)_*$ is dual to the subalgebra of the Steenrod algebra generated by $P^1$. Let $W_l$ be the indecomposable $P(1)_*$-comodule of length $l$. The following lemma indicates how tensor products decompose:
\begin{lem}
\label{algebraic-tensor-decomposition}
Given $1\leq r\leq s\leq p$,
\[W_r\otimes W_s = \bigoplus_{i=c+1}^{r} \susp^{2r-2i}W_p\oplus\bigoplus_{i=1}^c \susp^{2r-2i}W_{s-r+2i-1}\]
where
\[
c=\begin{cases}
r & \text{if $r+s\leq p$}\\
p-s & \text{if $r+s\geq p$}
\end{cases}
\]
and the first sum is empty if $r+s\leq p$.
\end{lem}
\begin{proof}
According to \cite[Theorem 1]{tensor-decomposition}, if $V_l$ is the length $l$ representation of $\KCp$, replacing $W_l$ with $V_l$ everywhere in the formula gives the tensor decomposition for $V_r\otimes V_s$. By \Cref{rep-rings-iso}, this suffices.
\end{proof}

In other words, to decompose $W_r\otimes W_s$, first apply the corresponding $\sl_2$ decomposition rule for irreducible $\sl_2$-modules of dimension $r$ and $s$:
\[W_r\otimes W_s = W_{r-s+1} \oplus \cdots \oplus W_{r+s-1}\]
When $l$ is larger than $p$, there is no indecomposable $P(1)_*$ comodule named $W_{l}$. If $W_{p+l}$ shows up in the list for $l>0$, then replace $W_{p+l}\oplus W_{p-l}$ with $W_p^{\oplus 2}$. This gives the decomposition rule.

\begin{cor}
\label{topological-tensor-decomposition}
Let $1\leq r\leq s\leq p$ and let $c$ be as in \Cref{algebraic-tensor-decomposition}. Then:
\[\EO\sm X_r\sm X_s\simeq \EO\sm \left(\bigvee_{i=c+1}^{r} \susp^{2r-2i} X_p\vee\bigvee_{i=1}^c \susp^{2r-2i}X_{s-r+2i-1}\right).\]
\end{cor}
\begin{proof}
$X_i\sm X_j$ is $2n$-sparse so by \Cref{2n-sparse-algebraic-EO-theory} $X_i\sm X_j$ has algebraic $\EO$ theory. Apply \Cref{splitting-from-HFp} to \Cref{algebraic-tensor-decomposition}.
\end{proof}

\begin{cor}
\label{alg-EO-mod-closed-under-smash}
If $M$ and $N$ are algebraic $\EO$-modules, then so is $M\sm_{\EO}N$.
\end{cor}

\begin{prop}
\label{alg-EO-mod-closed-under-sym}
Suppose that $M$ is an algebraic $\EO$-module and $i<p$ is an integer. Let $\Sym^i_{\EO}(M) = M^{\sm_{\EO}i}_{h\Sigma_i}$ be the $i$th symmetric power of $M$ relative to $\EO$. Then $\Sym^i(M)$ is algebraic and
\[\EEO_*\left(\Sym^i_{\EO}(X)\right)/\mfrak = \Sym^i\left(\EEO_*(X)/\mfrak\right)\]
where the right hand side is the symmetric power in $\KCp$-modules.
\end{prop}
It follows that the formula for symmetric powers of $\KCp$-modules determines the symmetric powers of $\EO$-module. Corollary 2.7 of \cite{symmetric-powers} gives a generating function for these symmetric powers which we quote as \Cref{symmetric-powers-generating-function}.

\begin{proof}
Using the binomial formula for the symmetric powers of a sum, the theorem reduces to the case of $M=\EO\sm X_l$. In this case $\Sym^i_{\EO}(\EO\sm X_l)=\EO\sm \Sym^i(X_l)$ so we need to show that $\Sym^i(X_l)$ has algebraic $\EO$ theory. If $E$ is a spectrum with trivial $G$ action and $X$ has a $G$-action, then $E\sm X_{hG}\simeq (E\sm X)_{hG}$, so there is a homotopy orbit spectral sequence $H_*^{\Sigma_i}(E_*(X_l^{\sm i}))\Rightarrow E_*\left(\Sym^i(X_l)\right)$. Since $i$ is $p$-locally invertible and there is a K\"unneth formula, we deduce that $E_*(\Sym^i(X_l))=\Sym^i(E_*(X_l))$ and $\HFp_*\left(\Sym^i(X_l)\right) = \Sym^i\left(\HFp_*X_l\right)$. Because $X_l$ is $2n$-sparse, so is $\Sym^i\left(\HFp_*X_l\right)$ and thus $\Sym^i(X_l)$ is $2n$-sparse and has algebraic $\EO$ theory.
\end{proof}

\section{\texorpdfstring{$Y_{2p}$}{Y\_{2p}} has Algebraic \texorpdfstring{$\EO$}{EO}-theory}
\label{sec:Y2p}
As a fun application of our theory, we show that $Y_{2p}$ has algebraic $\EO$ theory.
\begin{thm}
\label{Y2p-orientation}
Let $Y_{2p}=\OS{\BPn<1>}{2p}$. Suppose $Z$ is any spectrum with $\HFp_*(Z)\cong \HFp_*(Y_{2p})$ as $P(1)_*$-comodules. Then $Z$ has algebraic $\EO$ theory.
\end{thm}
By \Cref{alg-EO-from-homology-impl-Chern-determined} we deduce:
\begin{cor}
$\EO$-orientability of complex bundles over $Y_{2p}$ is Chern-determined.
\end{cor}

To understand $\EO\sm Y_{2p}$ we first need to compute $\HFp^*(Y_{2p})$ as a $P(1)^*$-module. First we check:
\begin{lem}
\label{HFpY2p-Sym}
As a $P(1)^*$-module, $\HFp^*(Y_{2p}) = \Fp[c_p, c_{p+n}, c_{p+2n},\ldots]$ with the $P(1)^*$-module structure:
\[
    P^1(c_{i})=(i+n)c_{i+n}
\]
\end{lem}
Note that this implies that $\Fp\{c_p, c_{p+n}, c_{p+2n},\ldots\}$ is a $P(1)^*$ submodule of $\HFp^*(Y_{2p})$ and so $\HFp^*(Y_{2p})=\Sym(\Fp\{c_p, c_{p+n}, c_{p+2n},\ldots\})$.  Furthermore $\Fp\{c_p, c_{p+n}, c_{p+2n},\ldots\}$ is the free $P(1)^*$-module on the generators $c_{p + pnk}$ for all $k$.
\begin{proof}
According to \cite{wilson}, the cohomology of $Y_{2p}$ is as indicated as an $\Fp$-algebra. There is a map of spectra $\phi\colon\BPo\to \ku$ including an Adams summand. Call the splitting map $\rho$. These induce maps between the loop spaces of $\BPo$ to the loop spaces of $\ku$. Because $\phi(v_1)=\beta_1^{p-1}$ we have a commutative diagram of infinite loop spaces:
\[\begin{tikzcd}
\hphantom{\CPinfty\times} Y_{2p} \rar[equals] &[-2em] \BPOS{2p} \rar[, "\phi"]\dar[xshift=-3pt,shorten <=-2pt,"v_1"'] & \OS{\ku}{2p}\dar[shorten <=-2pt, "\beta^{p-1}"]\\
          \CPinfty\times  Y_{2p} \rar[equals] & \BPOS{2} \rar[yshift=-3pt,"\phi"']\uar[shorten >=-2pt,dashed,xshift=3pt] &\lar[dashed,yshift=3pt,"\rho"'] \OS{\ku}{2}\mathrlap{{}=\BU}
\end{tikzcd}\]
By \cite[main theorem]{wilson}, $\OS{\BPo}{2}\simeq \CPinfty\times Y_{2p}$ so the map $\Ytp\to \OS{\BPo}{2}$ is a retract of spaces and the vertical dashed map exists. Thus, the map $Y_{2p}\to BU$ is a retract and we get a surjection $\HFp^*(BU)\to \HFp^*(Y_{2p})$. Since $Y_{2p}$ is $2p$-connective, this factors through $\HFp^*(BU)/(c_1,\ldots,c_{n})$ where $\HFp^*(BU)=\Fp[c_i]$. In $\HFp^*(BU)$ we have the formula:
\[P^1c_i = c_i\psi_{n} - c_{i+1}\psi_{n-2} + -\cdots -c_{i+n-1}\psi_1 + (i+n)c_{i+n}\]
Because $\psi_i\in (c_1,\ldots,c_n)$ for $i\leq n$, we deduce that $P^1(c_i)=(i+n)c_{i+n}$.
\end{proof}

The last input we need for this is \cref{symmetric-powers-P(1)} which computes the following symmetric powers formula:
\[\Sym^{k}(W_p) =
    \begin{cases}
    W_p^d  & p \nmid k\\
    W_1 \oplus W_p^d & p | k
    \end{cases}
\]

If $M=\Fp[C_p]\{x_1,\ldots,x_d\}$ is a free $C_p$-module then the trivial summands in $\Sym(M)$ are generated by $(x_1^{k_1}\cdots x_d^{k_d})^p$.

\begin{proof}[Proof of \Cref{Y2p-orientation}]
We showed in \cref{HFpY2p-Sym} that
\[\HFp^*(Y_{2p}) = \Sym(\Fp\{c_{p+nk}\})\]
as $P(1)^*$-representations. By \Cref{symmetric-powers-P(1)}, this splits into a sum of free modules and trivial modules where the trivial modules are generated by elements of the form $(c_{p+npk_1}\cdots c_{p+npk_i})^p$, so these all live in degrees congruent to zero mod $2p$. By \Cref{frees-split}, we can express $\EO\sm Y_{2p} = M \vee \bigwsum_{\infty} E $ where $M$ has cells in degrees $0$ mod $2p$.
By \Cref{2p-sparse-algebraic-EO-theory}, $M$ splits as a sum of $\EO$'s. On the other hand, $\HFp^*(MY_{2p})\cong \HFp^*(Y_{2p})$ as $P(1)^*$-representations because $P^1(u)\in \HFp^{2p-2}(\Mf)=0$.
The same logic applies to $\EO\sm \Mf$ and shows that it has the same splitting as $\EO\sm Y_{2p}$. Matching the summands gives a Thom isomorphism. The composite $\Mf\to \EO\sm \Mf \to \EO\sm \susp^{\infty}_+Y_{4p-4} \to \EO$ is a unital map $\Mf\to \EO$.
\end{proof}

\appendix
\section{Symmetric Powers of \texorpdfstring{$P(1)^*$}{P(1)*} Modules}
In this appendix, we show that the representation rings of $\Fp[C_p]$ and $P(1)^*$ are isomorphic, that the symmetric powers are the same for each, and discuss a formula for the symmetric powers.

For $A$ a Hopf algebra, let $\Rep(A)$ be the representation tensor category of $A$ and let $R(A)$ be the representation ring.
\begin{lem}
\label{rep-rings-iso}
Let $V_l$ be the length $l$ cyclic module over $\Fp[C_p]$ and let $W_l$ be the length $l$ cyclic module over $P(1)^*$. The map $\phi\colon R(\Fp[C_p])\to R(P(1)^*)$ sending $V_l\mapsto W_l$ is an isomorphism of representation rings.
\end{lem}
There is no lift of $\phi$ to a functor because that would imply an isomorphism of Hopf algebras $\Fp[C_p]\cong P(1)^*$ by Tannakian reconstruction.

\begin{proof}
It is clear that $\phi$ is an isomorphism of the underlying graded abelian groups, where the tensor product is forgotten. We need to show $\phi$ respects decompositions of tensor products. By direct computation, it is not hard to show that $V_2\otimes V_l \cong V_{l-1}\oplus V_{l+1}$ for $1<l < p$ so $V_2$ tensor generates the category of $\Fp[C_p]$-modules (see \cite[Theorem 1]{tensor-decomposition}). This implies that $V_2$ is a generator for $R(\Fp[C_p])$. To check that a map is a ring homomorphism, it suffices to check that $f(xy)=f(x)f(y)$ for $x$ and $y$ each pair of elements in a generating set. In this case, our generating set has one element, so it suffices to show $\phi([V_2]^2) = \phi([V_2])^2$. The formula is given by
\[\phi([V_2]^2)=\phi([V_1]+ [V_3])=\phi([V_1])+\phi([V_3])=[W_1]+[W_3] = [W_2]^2 = \phi([V_2])^2\qedhere\]
\end{proof}

Let $A_q$ be the Hopf algebra $\Fp[q,t]/(t^p)$ with the coaction $\Delta(t) = t\otimes 1 + 1\otimes t + t\otimes t$. Let $\Repqf$ be the full subcategory of $\Rep(A_q)$ spanned by representations that are free as $\Fp[q]$-modules. Consider the following diagram of Hopf algebras:
\[\begin{tikzcd}
& A_q\dlar["\theta_1"']\drar["\theta_0"] & \\
\Fp[C_p] && P(1)^*
\end{tikzcd}\]
The algebras $A_1$ and $A_0$ are specializations of $A_q$: $A_1 = A_q/(q-1) = \Fp[C_p]$ and $A_0=A_q/(q) = P(1)^*$.  It is fun to note that this diagram is the kernel of Verschiebung on the following diagram of formal groups:
\[\begin{tikzcd}
 & \ku\dlar\drar &&&& \ku^*(\CPinfty)\dlar \drar \\
\KU & & \HFp && \KU^*(\CPinfty) & & \HFp^*(\CPinfty)
\end{tikzcd}\]

\begin{lem}
The map $\theta_1^*\colon\Rqf(A_q)\to R(A_1)$ induced by tensoring down along $A_q\to A_1$ admits a section $\alpha$ such that $\theta_0\circ\alpha = \phi$.
\end{lem}

\begin{proof}
We set $\alpha(V_i)=U_i$. Since $V_2$ generates $R(A_1)$, it suffices to check that $\alpha([V_2]^2)=[U_2]^2$. Let $\{x_1,x_2\}$ be a basis for $U_2$ so that the action is given by $t(x_1)=x_2$. Then $U_2\otimes U_2$ has basis $\{x_{1}\otimes x_{1},x_1\otimes x_2, x_2\otimes x_1, x_2\otimes x_2\}$. The vector $x_1\otimes x_2 - x_2\otimes x_1$ is fixed and generates a $U_1$. On the other summand $t(x_1\otimes x_1)=x_2\otimes x_1 + x_1\otimes x_2 + qx_2\otimes x_2$. Also, $t(x_2\otimes x_1 + x_1\otimes x_2) = 2x_2\otimes x_2$. The matrix representation on $\{x_1\otimes x_1, x_2\otimes x_1 + x_1\otimes x_2, x_2\otimes x_2\}$ is given by:
\[
\begin{pmatrix}
1 & 0 & 0\\
1 & 1 & 0\\
u & 2 & 1
\end{pmatrix}
\]
Basic linear algebra shows that this is conjugate to a length 3 Jordan block, so $\{x_1\otimes x_1, x_2\otimes x_1 + x_1\otimes x_2, x_2\otimes x_2\}\cong  U_3$. Thus, $U_2\otimes U_2\cong U_1\oplus U_3$ and so $\alpha([V_2])^2= [U_2]^2 = [U_1]+[U_3]=\alpha([V_1]+[V_3])=\alpha([V_2]^2)$.
\end{proof}

\begin{thm}
\label{phi-is-a-sym-homomorphism}
Let $F$ be a not necessarily multiplicative, not necessarily additive natural transformation from the identity functor on the category of tensor categories to itself. For any tensor category $C$, $F$ gives a map of sets $R(F)\colon R(C)\to R(C)$. Then the following diagram commutes:
\[\begin{tikzcd}
 R(A_1) \rar["\phi"]\dar["R(F)"]&  R(A_0)\dar["R(F)"]\\
 R(A_1)\rar["\phi"] & R(A_0)
\end{tikzcd}\]
\end{thm}
We are only going to apply this when $F$ is one of the functors $\Sym^n$. In that case, it says that $\phi\colon R(A_1)\to R(A_0)$ is homomorphic for $\Sym^n$. I would like to say that $\phi$ is an isomorphism of $\Lambda$-rings, but neither the domain nor the codomain is actually a $\Lambda$-ring.
\begin{proof}
Suppose that $V\in \Rep(A_1)$, that $U\in \Rep(A_q)$ and $W=\phi(V)\in\Rep(A_0)$. Suppose that $F(U)$ has indecomposable decomposition $\bigoplus_{i=1}^{n} a_iU_i$. Then because $F$ is natural and $\phi$ commutes with $\oplus$,
\begin{align*}
F(V)=F(\theta_1(U))=\theta_1(F(U))=\theta_1\left(\bigoplus_{i=1}^{n} a_iU_i\right) = \bigoplus_{i=1}^n a_iV_i
\end{align*}
and likewise
\begin{align*}
F(W)=F(\theta_0(U))=\theta_0(F(U))=\theta_0\left(\bigoplus_{i=1}^{n} a_iU_i\right) = \bigoplus_{i=1}^n a_iW_i
\end{align*}
We see that $\phi(F(V))=F(\phi(V))$ so $\phi$ is homomorphic for $F$ as desired.
\end{proof}

We still have the issue of computing symmetric powers for $\Fp[C_p]$-modules. Hughes and Kemper \cite{symmetric-powers} compute the symmetric powers for $\Fp[C_p]$-modules. They have the following results:
\begin{lem}[{\cite[Lemma 2.3 and Theorem 2.4]{symmetric-powers}}]
Let $K=\Fp$ and let $\RKC$ be the representation ring of $\Fp[C_p]$-modules. The ring $\RKC$ is generated by $V_1,\ldots, V_p$. Define
\[
    R(\Fp[C_p])[\mu] = R(\Fp[C_p])[t]/(t^2-V_2t + 1).
\]
Then $\mu$ is invertible,
\[V_n = \frac{\mu^n-\mu^{-n}}{\mu-\mu^{-1}} = \sum_{j=0}^{n-1}\mu^{n-1-2j}\]
and $R(\Fp[C_p])[\mu]\cong \Z[\mu]/f(\mu)$ where
\[f(x) = \frac{(x-1)(x^{2p}-1)}{x+1}.\]
\end{lem}

\begin{thm}[{\cite[Corollary 2.7]{symmetric-powers}}]
\label{symmetric-powers-generating-function}
Let $\sigma_t(V)\in R(K[C_p])[\mu]\ps{t}$ be the generating function
\[\sum_{d=0}^{\infty} (\Sym^d V)t^d.\] Then
\[\sigma_t(V_{n+1})=\prod_{j=0}^n(1-\mu^{n-2j}t)^{-1}\pmod{t^p}\]
\end{thm}

\begin{thm}[{\cite[Theorem 2.11]{symmetric-powers}}]
\label{symmetric-powers-periodicity}
\[\Sym^{r+p}V_n \cong \Sym^r V_n \oplus V_p^{\oplus d}\]
\end{thm}

Together, we can use these to compute the case that we need:
\begin{thm}
\[\Sym^{k} V_p =
    \begin{cases}
    V_p^d  & p \nmid k\\
    V_1 \oplus V_p^d & p | k
    \end{cases}
\]
\label{symmetric-powers-Cp}
\end{thm}

\begin{proof}
Let $R=R(\Fp[C_p])/(V_p)$. In $R$ we want to show that
\[\Sym^{k} V_p =
    \begin{cases}
    0  & p \nmid k\\
    V_1 & p | k
    \end{cases}
\]

In $R(\Fp[C_p])$, $V_p = \frac{\mu^{p}-\mu^{-p}}{\mu-\mu^{-1}}$ and $\mu$ is a unit, so it is equivalent to quotient by $\frac{\mu^{2p}-1}{\mu^2-1}$. This divides the polynomial $f(\mu) = \frac{(x-1)(x^{2p}-1)}{x+1}$ so $R\cong \Z[\mu]/g(\mu)$ where $g(x)=\frac{x^{2p}-1}{x^2-1}=\Psi_{2p}$ where $\Psi_{2p}$ is the cyclotomic polynomial. Thus $\mu$ is a primitive $2p$th root of unity in $R$ and $\psi_r(\mu^p,\ldots,\mu^{2-p})=0$ for $1\leq r < p$. \Cref{symmetric-powers-generating-function} says the generating function for $\Sym(V_p)$ is given by
\begin{align*}
\sigma_t(V_p) &= \prod_{j=0}^{p-1} (1-\mu^{p-2j}t)^{-1}\pmod{t^p}\\
& = \sum_{i=0}^{p-1} h_i(\mu^p,\ldots,\mu^{2-p})
\end{align*}
and some multiple of $nh_i$ is generated by the $\psi_i$. We deduce that some multiple $nh_i(\mu^p,\ldots,\mu^{2-p})=0$ and because $R$ is torsion free this implies $h_i=0$. Hence, $\Sym(V_p) = 1 \pmod{p,V_p}$. By \Cref{symmetric-powers-periodicity}, we are done.
\end{proof}

Combining \Cref{phi-is-a-sym-homomorphism} and \Cref{symmetric-powers-Cp} gives:
\begin{cor}
\label{symmetric-powers-P(1)}
\[\Sym^{k} W_p =
    \begin{cases}
    W_p^d  & p \nmid k\\
    W_1 \oplus W_p^d & p | k
    \end{cases}
\]
\end{cor}

\bibliography{EO-splittings}
\bibliographystyle{plain}
\end{document}